\documentclass[1p,fleqn,11pt]{elsarticle}

\usepackage{amsmath,amsthm,amsfonts,amscd}
\usepackage{eucal} 	 	
\usepackage{verbatim}      	
\usepackage{makeidx}       	
\usepackage{./setup/psfig}         	
\usepackage{epsfig}         	
\usepackage{./setup/citesort}         	%
\usepackage{url}		

\newtheorem{thm}{Theorem}[section]

\newtheorem{prop}[thm]{Proposition}

\theoremstyle{definition}

\theoremstyle{remark}
\newtheorem{rem}{Remark}[section]

\usepackage{amsmath}
\usepackage{amssymb}
\usepackage{amsfonts}
\usepackage{amsthm}
\usepackage{amscd}
\usepackage[normalem]{ulem}
\usepackage{graphicx}
\usepackage[hang,small]{caption}
\usepackage[labelformat=simple]{subcaption}

\usepackage{wrapfig}
\usepackage{multirow}
\usepackage{xparse}
\usepackage{environ}
\usepackage{stackengine,wasysym}
\usepackage{mathtools}
\mathtoolsset{showonlyrefs}
\usepackage{todonotes}
\makeatletter
\renewcommand{\todo}[2][]{\tikzexternaldisable\@todo[#1]{#2}\tikzexternalenable}
\makeatother

\usepackage{tensor}
\usepackage{varwidth,calc}

\usepackage{pgfplotstable}
\usepackage{pgfplots}
\pgfplotsset{compat=1.14}
\pgfplotsset{
	colormap={my basis colormap}{
		rgb255=(0, 114, 189);
		rgb255=(54, 106, 148);
		rgb255=(108, 98, 107);
		rgb255=(163, 91, 66);
		rgb255=(217, 83, 25);
	}
}
\pgfplotsset{
	colormap={my parula}{
		rgb255=(53.0655, 42.406, 134.9460);
		rgb255=(20.2336, 132.6061, 211.9511);
		rgb255=(55.5463, 184.8859, 157.9118);
		rgb255=(208.7187, 186.8470,  89.1966);
		rgb255=(248.9565, 250.6905, 13.7190);
	}
}
\pdfpageattr {/Group << /S /Transparency /I true /CS /DeviceRGB>>}
\usepackage{tikz}
\usepgfmodule{oo}
\usetikzlibrary{calc}
\usetikzlibrary{shapes}
\usetikzlibrary{plotmarks}
\usetikzlibrary{backgrounds}
\usetikzlibrary{decorations.pathmorphing,decorations.pathreplacing,decorations.markings}
\usetikzlibrary{arrows.meta,arrows,fit,matrix,positioning,shapes.geometric,positioning}
\usetikzlibrary{external}
\usetikzlibrary{cd}
\AtBeginEnvironment{tikzcd}{\tikzexternaldisable}
\AtEndEnvironment{tikzcd}{\tikzexternalenable}

\usepackage{appendix}
\usepackage{chngcntr}
\usepackage{etoolbox}
\usepackage{apptools}
\AtAppendix{\counterwithin{theorem}{section}}
\AtBeginEnvironment{subappendices}{%
	\chapter*{Appendix}
	\addcontentsline{toc}{chapter}{Appendices}
}

\usepackage{breakcites}
\usepackage[inline]{enumitem}

\usepackage{scalerel,stackengine,wasysym}

\usepackage{accents}

\usepackage{algorithm}
\usepackage[noend]{algpseudocode}

\DeclareMathAlphabet{\mathpzc}{OT1}{pzc}{m}{it}

\usepackage{color}
\definecolor{lightgray}{gray}{0.80}
\definecolor{lightgray}{gray}{0.80}

\usepackage{mdframed}

\usepackage[titles]{tocloft}
\cftsetindents{figure}{0em}{3.5em}
\cftsetindents{table}{0em}{3.5em}

\usepackage{lscape}
\renewcommand\geq\geqslant
\renewcommand\leq\leqslant

\newcommand{\gi}{lower-acyclic}
\newcommand{\ginoun}{lower-acyclicity}
\newcommand{\ginounC}{Lower-acyclicity}

\usepackage{titlesec}
\titleformat{\chapter}[block]
{\normalfont\huge\bfseries}{\thechapter.}{1em}{\huge}
\titlespacing*{\chapter}{0pt}{-19pt}{0pt}

\usepackage[T3,T1]{fontenc}

\newcommand{\includetikz}[2]{%
	\includegraphics{#2/images/#1}
}

\theoremstyle{plain}
\newtheorem{theorem}{Theorem}[section]
\newtheorem{proposition}[theorem]{Proposition}
\newtheorem{lemma}[theorem]{Lemma}
\newtheorem{corollary}[theorem]{Corollary}

\theoremstyle{remark}
\newtheorem{remark}[theorem]{Remark}

\newtheorem{examplex}[theorem]{Example}
\newenvironment{example}
{\pushQED{\qed}\examplex}
{\popQED\endexamplex}

\theoremstyle{definition}
\newtheorem{definition}[theorem]{Definition}

\usepackage{mathtools}

\newcommand{\mbf}[1]{{\boldsymbol{#1}}}
\newcommand{\RR}{{\mathbb{R}}}

\newcommand{\PP}{{\mathcal{P}}}

\newcommand{\ZZ}{{\mathbb{Z}}}
\newcommand{\ZZP}{\mathbb{Z}_{\geq0}}

\newcommand{\dimwp}[1]{\dim{\left(#1\right)}}

\newcommand{\moplus}{\mathop{\oplus}}

\newcommand{\isomorphic}{\cong}

\stackMath
\newcommand\reallywidehat[1]{%
	\savestack{\tmpbox}{\stretchto{%
			\scaleto{%
				\scalerel*[\widthof{\ensuremath{#1}}]{\kern-.6pt\bigwedge\kern-.6pt}%
				{\rule[-\textheight/2]{1ex}{\textheight}}
			}{\textheight}%
		}{0.5ex}}%
	\stackon[1pt]{#1}{\tmpbox}%
}
\stackMath
\newcommand\reallywidecheck[1]{%
	\savestack{\tmpbox}{\stretchto{%
			\scaleto{%
				\scalerel*[\widthof{\ensuremath{#1}}]{\kern-.6pt\bigvee\kern-.6pt}%
				{\rule[-\textheight/2]{1ex}{\textheight}}
			}{\textheight}%
		}{0.5ex}}%
	\stackon[1pt]{#1}{\tmpbox}%
}

\ExplSyntaxOn
\NewEnviron{bmatrixT}
{
	\marine_transpose:V \BODY
}

\int_new:N \l_marine_transpose_row_int
\int_new:N \l_marine_transpose_col_int
\seq_new:N \l_marine_transpose_rows_seq
\seq_new:N \l_marine_transpose_arow_seq
\prop_new:N \l_marine_transpose_matrix_prop
\tl_new:N \l_marine_transpose_last_tl
\tl_new:N \l_marine_transpose_body_tl

\cs_new_protected:Nn \marine_transpose:n
{
	\seq_set_split:Nnn \l_marine_transpose_rows_seq { \\ } { #1 }
	\int_zero:N \l_marine_transpose_row_int
	\prop_clear:N \l_marine_transpose_matrix_prop
	\seq_map_inline:Nn \l_marine_transpose_rows_seq
	{
		\int_incr:N \l_marine_transpose_row_int
		\int_zero:N \l_marine_transpose_col_int
		\seq_set_split:Nnn \l_marine_transpose_arow_seq { & } { ##1 }
		\seq_map_inline:Nn \l_marine_transpose_arow_seq
		{
			\int_incr:N \l_marine_transpose_col_int
			\prop_put:Nxn \l_marine_transpose_matrix_prop
			{
				\int_to_arabic:n { \l_marine_transpose_row_int }
				,
				\int_to_arabic:n { \l_marine_transpose_col_int }
			}
			{ ####1 }
		}
	}
	\tl_clear:N \l_marine_transpose_body_tl
	\int_step_inline:nnnn { 1 } { 1 } { \l_marine_transpose_col_int }
	{
		\int_step_inline:nnnn { 1 } { 1 } { \l_marine_transpose_row_int }
		{
			\tl_put_right:Nx \l_marine_transpose_body_tl
			{
				\prop_item:Nn \l_marine_transpose_matrix_prop { ####1,##1 }
				\int_compare:nF { ####1 = \l_marine_transpose_row_int } { & }
			}
		}
		\tl_put_right:Nn \l_marine_transpose_body_tl { \\ }
	}
	\begin{bmatrix*}[r]
		\l_marine_transpose_body_tl
	\end{bmatrix*}
}
\cs_generate_variant:Nn \marine_transpose:n { V }
\cs_generate_variant:Nn \prop_put:Nnn { Nx }
\ExplSyntaxOff

\makeatletter
\catcode`! 3
\def\Transpose #1{\romannumeral0\expandafter
	\Mar@Transpose@a\romannumeral`^^@\Mar@DoOneRow #1\\!\\}

\def\Mar@DoOneRow #1\\{\Mar@DoOneRow@a {}#1&^^@&}%

\def\Mar@DoOneRow@a #1#2&{%
	\if^^@\detokenize{#2}\expandafter\@gobble\fi
	\Mar@DoOneRow@a {#1#2\\}%
}%

\def\Mar@Transpose@a #1#2\\{\ifx!#2\expandafter\Mar@FinishTranspose\fi
	\expandafter\Mar@Transpose@b\romannumeral`^^@\Mar@DoOneRow@a {}#2&^^@&#1}

\def\Mar@Transpose@b #1#2^^@\\{\Mar@Join {}#2^^@!#1}

\def\Mar@Join #1#2\\#3!#4\\%
{\if^^@\detokenize{#3}\expandafter\Mar@EndJoin\fi
	\Mar@Join {#1#2&#4\\}#3!}%

\def\Mar@EndJoin\Mar@Join #1^^@!^^@\\{\Mar@Transpose@a {#1^^@\\}}

\def\Mar@FinishTranspose
#1&^^@&#2\\^^@\\{ #2}

\catcode`! 12
\makeatother

\newcommand{\face}{\sigma}
\newcommand{\edge}{\tau}
\newcommand{\vertex}{\gamma}

\newcommand{\interior}[1]{\accentset{\circ}{#1}}
\newcommand{\mesh}{\mathcal{T}}
\newcommand{\meshF}{\mesh_2}
\newcommand{\meshE}{\mesh_1}
\newcommand{\meshV}{\mesh_0}
\newcommand{\meshInterior}{\interior{\mesh}}
\newcommand{\meshInteriorE}{\meshInterior_1}
\newcommand{\meshInteriorV}{\meshInterior_0}

\newcommand{\boundary}{\partial}
\newcommand{\domain}{\Omega}
\newcommand{\domainInterior}{\interior{\domain}}

\newcommand{\meshEH}{\tensor*[^h]{\mesh}{_1}}
\newcommand{\meshEV}{\tensor*[^v]{\mesh}{_1}}

\newcommand{\meshInteriorEV}{\tensor*[^v]{\meshInterior}{_1}}

\newcommand{\ncells}[1]{\mathfrak{t}_{#1}}

\newcommand{\degreeu}{m}

\newcommand{\smooth}{r}
\newcommand{\bsmooth}{\mbf{r}}

\newcommand{\splSpace}{\mathcal{R}}

\newcommand{\bsmoothr}{{\mbf{s}}}

\newcommand{\elem}[1]{[#1]}
\newcommand{\faceE}[1]{\elem{\face_{#1}}}
\newcommand{\edgeE}[1]{\elem{\edge_{#1}}}
\newcommand{\vertexE}[1]{\elem{\vertex_{#1}}}

\newcommand{\ideal}[1][I]{\mathfrak{#1}}
\newcommand{\idealComplex}{\mathcal{I}}
\newcommand{\constantComplex}{\mathcal{C}}
\newcommand{\quotientComplex}{\mathcal{Q}}

\newcommand{\euler}[1]{\chi \left( #1 \right)}

\newcommand{\smin}[1]{M(#1)}

\makeatletter
\def\ps@pprintTitle{%
	\let\@oddhead\@empty
	\let\@evenhead\@empty
	\def\@oddfoot{}%
	\let\@evenfoot\@oddfoot}
\makeatother

\usepackage[hidelinks]{hyperref}

\begin{document}
	
	\title{Counting the dimension of splines of mixed smoothness:\\
		A general recipe, and its application to meshes of arbitrary topologies}
	\corref{cor1}
	\author[delft]{Deepesh Toshniwal}
	\ead{d.toshniwal@tudelft.nl}	
	\cortext[cor1]{Corresponding author}
	\address[delft]{Delft Institute of Applied Mathematics, Delft University of Technology, the Netherlands}
	\author[csu]{Michael DiPasquale}
	\ead{michael.dipasquale@colostate.edu}	
	\address[csu]{Department of Mathematics, Colorado State University, United States of America}

	\begin{abstract}
		In this paper we study the dimension of bivariate polynomial splines of mixed smoothness on polygonal meshes.
		Here, ``mixed smoothness'' refers to the choice of different orders of smoothness across different edges of the mesh.
		To study the dimension of spaces of such splines, we use tools from Homological Algebra.
		These tools were first applied to the study of splines by Billera (1988).
		Using them, estimation of the spline space dimension amounts to the study of the generalized Billera-Schenck-Stillman complex for the spline space.
		In particular, when the homology in positions one and zero of this complex are trivial, the dimension of the spline space can be computed combinatorially.  We call such spline spaces ``\gi{}.''
		In this paper, starting from a spline space which is \gi, we present sufficient conditions that ensure that the same will be true for the spline space obtained after relaxing the smoothness requirements across a subset of the mesh edges.
		This general recipe is applied in a specific setting: meshes of arbitrary topologies.  We show how our results can be used to compute the dimensions of spline spaces on triangulations, polygonal meshes, and T-meshes with holes.
	\end{abstract}
	
	\begin{keyword}
		splines \sep polygonal meshes with holes \sep dimension formula \sep mixed smoothness
	\end{keyword}
	\maketitle
	
	\section{Introduction}\label{sec:introduction}

Piecewise-polynomial functions called splines are foundational pillars that support modern computer-aided geometric design \cite{farin2002handbook}, numerical analysis \cite{iga-book}, etc.
These functions are defined on polyhedral partitions of $\RR^n$.
Their restriction to any polyhedron's interior is a polynomial, and these  polynomial pieces are constrained to join with some desired smoothness across hyperplanes supporting the intersections of neighbouring polyhedra.
Here, we study bivariate spline spaces -- i.e., $n = 2$ -- of mixed smoothness -- i.e., different orders of smoothness constraints are imposed across different edges of the partition.
From the perspective of approximation with splines, mixed smoothness is particularly interesting for capturing local, non-smooth (or even discontinuous) features in the target function; e.g., for shock-capturing in fluid dynamics, or for modelling smooth geometries with localized creases.
In particular, we study how the dimension of such spline spaces can be computed.

Computing the dimension of spline spaces is a highly non-trivial task in general for splines in more than one variable.  Initiated by Strang~\cite{strang1973piecewise,strang1974dimension}, this is by now a classical topic in approximation theory and has been studied in a wide range of planar settings; e.g., on triangulations, polygonal meshes, and T-meshes \cite{schumaker1984bounds,alfeld1987dimension,billera1988homology,schenck1997local,schenck1997family,mcdonald_schenck_09,dipasquale2018dimension,toshniwal2019polynomial,mourrain2014dimension,toshniwal_polynomial_2019}.
Non-polynomial spline spaces have also been studied in the same vein; e.g., \cite{bracco_generalized_2016}.

In the present paper, instead of initiating the study of mixed-smoothness splines from scratch, we study them in relation to a proper subspace for which the dimension-computation problem is well-understood.
Several conceptually similar approaches have been recently formulated, inspired by applications of splines in numerical analysis and geometric modelling.
For instance, this approach was adopted to study splines on locally subdivided meshes in \cite{schenck_subdivision_2018}; to study splines with local polynomial-degree adaptivity in \cite{toshniwal2019polynomial,toshniwal_polynomial_2019}; and to study mixed-smoothness splines on T-meshes in \cite{toshniwal_mixed_2019}.

More specifically, we derive sufficient conditions that help describe mixed-smoothness spline spaces as \emph{lower-acyclic}, i.e., as spaces for which the dimension can be computed combinatorially using only local geometric information.
Working on a polygonal mesh in $\RR^2$, we start from a spline space $\splSpace^\bsmooth$ whose members are constrained to be at least $\bsmooth(\edge)$ smooth across edge $\edge$ of the mesh.
Then, given that $\splSpace^\bsmooth$ is lower-acyclic, we derive sufficient conditions for $\splSpace^\bsmoothr \supseteq \splSpace^\bsmooth$ to be lower-acyclic, where $\bsmoothr(\edge) \leq \bsmooth(\edge)$ for all edges $\edge$.
We use methods from homological algebra to derive these results; see Section \ref{sec:smoothness_reduction}.

The sufficient conditions derived are highly general and are applicable to a wide variety of non-standard spline spaces.
In order to examine the conditions in practice, we narrow our focus down to a specific application: dimension computation for spline spaces on meshes of arbitrary topologies; e.g., see the figures below:
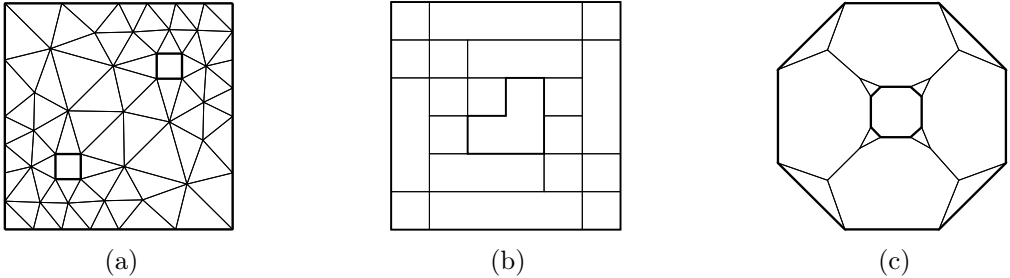
\begin{figure}[h]
	\centering
	\subcaptionbox{}[0.32\textwidth]{\resizebox{0.2\textwidth}{!}{\includetikz{tri_2holes}{./tikz}}}
	\subcaptionbox{}[0.32\textwidth]{\resizebox{0.2\textwidth}{!}{\includetikz{abstract-figure0}{./tikz}}}
	\subcaptionbox{}[0.32\textwidth]{\resizebox{0.2\textwidth}{!}{\begin{tikzpicture}[scale=3]
\tikzset{
	bThickness/.style={line width=#1\pgflinewidth},
	bThickness/.default={2},
}

\tikzset{
	eThickness/.style={line width=#1\pgflinewidth},
	eThickness/.default={0.5},
}

\begin{scope}
\coordinate (v0) at (0.92, 0.38) {};
\coordinate (v1) at (0.38, 0.92) {};
\coordinate (v2) at (-0.38, 0.92) {};
\coordinate (v3) at (-0.92, 0.38) {}; 
\coordinate (v4) at (-0.92, -0.38) {};
\coordinate (v5) at (-0.38, -0.92) {}; 
\coordinate (v6) at (0.38, -0.92) {};
\coordinate (v7) at (0.92, -0.38) {};
\coordinate (v8) at (0.54, 0.54) {};
\coordinate (v9) at (-0.52, 0.54) {};
\coordinate (v10) at (-0.52, -0.52) {};
\coordinate (v11) at (0.54, -0.52) {}; 
\coordinate (v12) at (0.31, 0.31) {}; 
\coordinate (v13) at (-0.26, 0.31) {}; 
\coordinate (v14) at (-0.26, -0.26) {};
\coordinate (v15) at (0.31, -0.26) {};
\coordinate (v16) at (0.24, 0.16) {}; 
\coordinate (v17) at (0.15, 0.24) {};
\coordinate (v18) at (-0.090, 0.24) {};
\coordinate (v19) at (-0.17, 0.16) {}; 
\coordinate (v20) at (-0.17, -0.088) {};
\coordinate (v21) at (-0.091, -0.17) {};
\coordinate (v22) at (0.15, -0.17) {};
\coordinate (v23) at (0.24, -0.089) {};

\draw[eThickness] (v0.center) -- (v1.center) -- (v8.center) -- cycle;
\draw[eThickness] (v2.center) -- (v3.center) -- (v9.center) -- cycle;
\draw[eThickness] (v4.center) -- (v5.center) -- (v10.center) -- cycle;
\draw[eThickness] (v6.center) -- (v7.center) -- (v11.center) -- cycle;
\draw[eThickness] (v15.center) -- (v22.center) -- (v23.center) -- cycle;
\draw[eThickness] (v12.center) -- (v16.center) -- (v17.center) -- cycle;
\draw[eThickness] (v13.center) -- (v18.center) -- (v19.center) -- cycle;
\draw[eThickness] (v14.center) -- (v20.center) -- (v21.center) -- cycle;
\draw[eThickness] (v0.center) -- (v8.center) -- (v12.center) -- (v16.center) -- (v23.center) -- (v15.center) -- (v11.center) -- (v7.center) -- cycle;
\draw[eThickness] (v1.center) -- (v2.center) -- (v9.center) -- (v13.center) -- (v18.center) -- (v17.center) -- (v12.center) -- (v8.center) -- cycle;
\draw[eThickness] (v3.center) -- (v4.center) -- (v10.center) -- (v14.center) -- (v20.center) -- (v19.center) -- (v13.center) -- (v9.center) -- cycle;
\draw[eThickness] (v5.center) -- (v6.center) -- (v11.center) -- (v15.center) -- (v22.center) -- (v21.center) -- (v14.center) -- (v10.center) -- cycle;

\draw[bThickness] (v16.center) -- (v17.center) -- (v18.center) -- (v19.center) -- (v20.center) -- (v21.center) -- (v22.center) -- (v23.center) -- cycle;
\draw[bThickness] (v0.center) -- (v1.center) -- (v2.center) -- (v3.center) -- (v4.center) -- (v5.center) -- (v6.center) -- (v7.center) -- cycle;
\end{scope}
\end{tikzpicture}}}
	\caption{We study the dimension of splines on polygonal meshes of arbitrary topologies in Section \ref{sec:examples}.
	Here, the mesh boundaries have been displayed in bold.}
\end{figure}

Such splines enable geometric modelling of and numerical analysis on arbitrary smooth surfaces \cite{toshniwal2017smooth}, and are very useful in applications.
We investigate the application of our results to the following particular cases; see Section \ref{sec:examples} for the details:
\begin{itemize}
	\item total-degree splines on triangulations and polygonal meshes containing holes;
	\item mixed bi-degree splines on T-meshes containing holes.
\end{itemize}
%
	\section{Preliminaries: splines, meshes and homology}\label{sec:preliminaries}

This section will introduce the relevant notation that we will use for working with polynomial splines on triangulations and T-meshes.

\subsection{Bivariate splines on planar meshes}\label{ss:meshes_and_splines}

\begin{definition}[Mesh]\label{def:mesh}
	A mesh $\mesh$ of $\RR^2$ is defined as:
	\begin{itemize}
		\item a finite collection $\meshF$ of polygons $\face$ that we consider as open sets of $\RR^2$ having non-zero measure, called $2$-cells or faces, together with
		\item a finite set $\meshE$ of closed segments $\edge$, called $1$-cells, which are edges of the (closure of the) faces $\sigma\in\meshF$, and
		\item the set $\meshV$, of vertices $\vertex$, called $0$-cells,  of the edges $\tau\in\meshE$,
	\end{itemize}
	such that the following properties are satisfied:
	\begin{itemize}
		\item $\face \in \meshF \Rightarrow$ the boundary $\boundary\face$ of $\sigma$ is a finite union of edges in $\meshE$,
		\item $\face, \face' \in \meshF \Rightarrow \face \cap \face' = \boundary\face \cap \boundary\face'$ is a finite union of edges in $\meshE \cup \meshV$, and,
		\item $\edge, \edge' \in \meshE \text{ with } \edge \neq \edge' \Rightarrow \edge\cap \edge' 
		\in \meshV$.
	\end{itemize}
	The domain of the mesh is assumed to be connected and is defined as $\domain := \cup_{\face\in\meshF}\face \subset \RR^2$.
\end{definition}


The closures of the mesh faces $\face$ will be denoted by $\overline{\face}$.
Edges of the mesh will be called interior edges if they intersect the interior of the domain of the mesh, $\domainInterior$.
Otherwise, they will be called boundary edges. 
The set of interior edges will be denoted by $\meshInteriorE$.
Similarly, if a vertex is in $\domainInterior$ it will be called an interior vertex, and a boundary vertex otherwise.
The set of interior vertices will be denoted by $\meshInteriorV$.
We will denote the number of $i$-cells with $\ncells{i} := \# \mesh_i$.

The first ingredient we need for defining polynomial splines on $\mesh$ are vector spaces of polynomials attached to each face of the mesh.
More precisely, to each face $\face$ of the mesh, we will assign a vector space of (total degree or bi-degree) polynomials denoted by $\PP_\face$,
\begin{equation}
	\mbf{\degreeu}~:~ \face \mapsto \PP_\face\;.
\end{equation}
If the closures of faces $\face$ and $\face'$ have non-empty intersection, then we will assume that $\PP_\face + \PP_\face'$ is either $\PP_\face$ or $\PP_\face'$.
Then, we can use $\PP_\face$ to assign vector spaces of polynomials to the edges and vertices of $\mesh$.
Denoting these by $\PP_\edge$ and $\PP_\vertex$, respectively, for $\edge \in \meshE$ and $\vertex \in \meshV$, we define them as follows,
\begin{equation}
	\PP_\edge := \sum_{\overline{\face} \supset \edge} \PP_\face\;,
	\qquad
	\PP_\vertex := \sum_{\overline{\face} \ni \vertex} \PP_\face\;.
	\label{eq:edge_vertex_spaces}
\end{equation}
The above assignment of vector spaces to faces, edges and vertices of $\mesh$ will be assumed to be fixed throughout this document.

The second and final ingredient that we need for defining splines on $\mesh$ is a smoothness distribution on its edges.
The objective of this paper is to study how the dimension of the space of splines on $\mesh$ (which will be defined shortly) changes with the smoothness distribution; the latter object is defined as follows.

\begin{definition}[Smoothness distribution]
	The map $\bsmooth : \meshE \rightarrow \ZZ_{\geq -1}$ is called a smoothness distribution if $\bsmooth(\edge) = -1$ for all $\edge \notin \meshInteriorE$.
\end{definition}

Using this notation, we can now define the spline space $\splSpace^\bsmooth$ that forms the object of our study.
From the following definition and the definition of $\bsmooth$, it will be clear that we are interested in obtaining highly local control over the smoothness of splines in $\splSpace^\bsmooth$, a feature that is missing from most of the existing literature.
This problem has been addressed in a recent paper \cite{toshniwal_mixed_2019}, but in a restriction setting where it is assumed that (a) $\mesh$ is a T-mesh and (b) $\PP_\face = \PP_{\face'}$ for any $\face, \face' \in \meshF$.
Thus, the setting of the present paper is much more general.

\begin{definition}[Spline space]
	The spline space $\splSpace^\bsmooth \equiv \splSpace^{\bsmooth}_{\mbf{\degreeu}}(\mesh)$ as
	\begin{equation}
	\begin{split}
	\splSpace^{\bsmooth}_{\mbf{\degreeu}}(\mesh) := \bigg\{f\colon &\forall \face \in \meshF~~f|_\face \in \PP_{\face}\;,\\
	&\forall \edge \in \meshInteriorE~~f \text{~is~} C^{\bsmooth(\edge)}\text{~smooth across~}\edge\bigg\}\;.
	\end{split}
	\end{equation}
\end{definition}

From the above definition, the pieces of all splines in $\splSpace^\bsmooth$ are constrained to meet with smoothness $\bsmooth(\edge)$ at an interior edge $\edge$.
We will use the following algebraic characterization of smoothness in this document.
\begin{lemma}[Billera \cite{billera1988homology}]\label{lem:smoothness}
	For $\face, \face' \in \meshF$, let $\face \cap \face' = \edge \in \meshInteriorE$, and consider a piecewise polynomial function equalling $p$ and $q$ on $\face$ and $\face'$, respectively.
	Then, this piecewise polynomial function is at least $\smooth$ times continuously differentiable across $\edge$ if and only if
	\begin{equation}
	\ell_\edge^{\smooth+1}~\big|~p - q\;,
	\end{equation}
	where $\ell_\edge$ is a non-zero linear polynomial vanishing on $\edge$.
\end{lemma}

In line with the above characterization and for each interior edge $\edge$, we define $\ideal^\bsmooth_\edge$ to be the vector subspace of $\PP_{\edge}$ that contains all polynomial multiples of $ \ell_\edge^{\bsmooth(\edge)+1}$; when $\bsmooth(\edge) = -1$, $\ideal^\bsmooth_\edge$ is simply defined to be $\PP_{\edge}$.
Similarly, for each interior vertex $\vertex$, we define $\ideal^\bsmooth_\vertex := \sum_{ \edge\ni \vertex} \ideal^\bsmooth_\edge$.
Here, we have suppressed the dependence of $\ideal^\bsmooth_\edge$ and $\ideal^\bsmooth_\vertex$ on $\mbf{\degreeu}$ to simplify the reading (and writing) of the text.

\subsection{Topological chain complexes}\label{ss:homological_interpretation}
Any spline $f \in \splSpace^\bsmooth$ is a piecewise polynomial function on $\mesh$.
We can explicitly refer to its piecewise polynomial nature by equivalently expressing it $\sum_{\face} \faceE{} f_\face$ with $f_\face := f|_\face$.
This notation makes it clear that the polynomial $f_\face$ is attached to the face $\face$ of $\mesh$.
Using this notation and Lemma \ref{lem:smoothness}, the spline space $\splSpace^\bsmooth$ can be equivalently expressed as the kernel of the map $\overline{\boundary}$,
\begin{equation}
\overline{\boundary} : \moplus_{\face \in \meshF} \faceE{} \PP_{\face} \rightarrow \moplus_{\edge \in \meshInteriorE} \edgeE{} \PP_{\edge}/\ideal^\bsmooth_\edge\;.
\end{equation}
defined by composing the boundary map $\boundary$ with the natural quotient map.

As a result of this observation, the spline space $\splSpace^\bsmooth$ can be interpreted as the top homology of a suitably defined chain complex $\quotientComplex^\bsmooth$,
\begin{equation}
	\begin{tikzcd}
		\quotientComplex^\bsmooth~:~ & \bigoplus\limits_{\face \in \meshF} \faceE{} \PP_{\face} \arrow[r] & \bigoplus\limits_{\edge \in \meshInteriorE} \edgeE{} \PP_{\edge}/\ideal^\bsmooth_\edge \arrow[r] & \bigoplus\limits_{\vertex \in \meshInteriorV} \vertexE{} \PP_{\vertex}/\ideal^\bsmooth_\vertex \arrow[r] & 0\;.
	\end{tikzcd}
\end{equation}
The above is the generalized Billera-Schenck-Stillman complex (abbreviated as gBSS).
The Billera-Schenck-Stillman complex was first introduced in \cite{billera1988homology,schenck1997family}, and the gBSS was first introduced in \cite{toshniwal2019polynomial,toshniwal_polynomial_2019}.
As in \cite{billera1988homology,schenck1997family,mourrain2014dimension}, we will study $\quotientComplex$ using the following short exact sequence of chain complexes,
\begin{equation}
	\begin{tikzcd}
		~ & ~ & 0 \arrow[d,""] & 0 \arrow[d,""] & ~ \\
		\idealComplex^\bsmooth~:~ & 0\arrow[r] \arrow[d]& \bigoplus\limits_{\edge \in \meshInteriorE} \edgeE{} \ideal^\bsmooth_\edge \arrow[r] \arrow[d] & \bigoplus\limits_{\vertex \in \meshInteriorV} \vertexE{} \ideal^\bsmooth_\vertex \arrow[r] \arrow[d] & 0 \\
		\constantComplex~:~ & \bigoplus\limits_{\face \in \meshF} \faceE{} \PP_{\face} \arrow[r] \arrow[d] & \bigoplus\limits_{\edge \in \meshInteriorE} \edgeE{} \PP_{\edge} \arrow[r] \arrow[d] & \bigoplus\limits_{\vertex \in \meshInteriorV} \vertexE{} \PP_{\vertex} \arrow[r] \arrow[d] & 0 \\
		\quotientComplex^\bsmooth~:~ & \bigoplus\limits_{\face \in \meshF} \faceE{} \PP_{\face} \arrow[r] & \bigoplus\limits_{\edge \in \meshInteriorE} \edgeE{} \PP_{\edge}/\ideal^\bsmooth_\edge \arrow[r] \arrow[d] & \bigoplus\limits_{\vertex \in \meshInteriorV} \vertexE{} \PP_{\vertex}/\ideal^\bsmooth_\vertex \arrow[r] \arrow[d] & 0 \\
		~ & ~ & 0 & 0 & ~
	\end{tikzcd}
\label{eq:complex}
\end{equation}

\begin{definition}[\ginounC{}  of gBSS]
	The complex $\quotientComplex^\bsmooth$ will be called \gi{} if its homologies in positions one and zero are trivial, i.e., if $H_1(\quotientComplex^\bsmooth) = 0 = H_0(\quotientComplex^\bsmooth)$.
\end{definition}

\ginounC{} of the gBSS is interesting precisely because it is a sufficient condition for ensuring that the dimension of $\splSpace^\bsmooth$ can be computed using only local combinatorial data, with the computation being unaffected by the global geometry of $\mesh$.

\begin{theorem}\label{thm:geom_indep_dim}
	If $\quotientComplex^\bsmooth$ is \gi{}, then the dimension of $\splSpace^\bsmooth$ can be combinatorially computed,
	\begin{equation}
	\begin{split}
		\dimwp{\splSpace^\bsmooth} &= \euler{\quotientComplex^\bsmooth}\;,
	\end{split}
	\end{equation}
	where $\euler{\quotientComplex^\bsmooth}$ is the Euler characteristic of the complex $\quotientComplex^\bsmooth$.
	Moreover,
	\begin{equation}
		H_0(\idealComplex^\bsmooth)
		\isomorphic
		H_0(\constantComplex)\;.
	\end{equation}
\end{theorem}
\begin{proof}
	The first claim follows from the definition of the Euler characteristic of $\quotientComplex^\bsmooth$,
	\begin{equation}
	\begin{split}
		\euler{\quotientComplex^\bsmooth} &= \dimwp{\quotientComplex^\bsmooth_2} - \dimwp{\quotientComplex^\bsmooth_1} + \dimwp{\quotientComplex^\bsmooth_0}\;,\\
		&= \dimwp{H_2(\quotientComplex^\bsmooth)} - \dimwp{H_1(\quotientComplex^\bsmooth)} + \dimwp{H_0(\quotientComplex^\bsmooth)}\;,\\
		&= \dimwp{\splSpace^\bsmooth} - \dimwp{H_1(\quotientComplex^\bsmooth)} + \dimwp{H_0(\quotientComplex^\bsmooth)}\;.
	\end{split}
	\end{equation}
	The second part of the claim follows from the long exact sequence of homologies implied by the short exact sequence of chain complexes in Equation \eqref{eq:complex},
	\begin{equation}
		\begin{tikzcd}
			\cdots \arrow[r]
			& H_1(\constantComplex) \arrow[r]
			& H_1(\quotientComplex^\bsmooth) \arrow[r]
			& H_0(\idealComplex^\bsmooth) \arrow[r] 
			& H_0(\constantComplex) \arrow[r]
			& H_0(\quotientComplex^\bsmooth) \arrow[r]
			& 0\;.
		\end{tikzcd}
	\end{equation}
\end{proof}
	
	\section{Spline space $\splSpace^{\bsmoothr} \supseteq \splSpace^\bsmooth$ of reduced regularity}\label{sec:smoothness_reduction}

We present our main results in this section.
In particular, we will relate the dimension of the spline space $\splSpace^\bsmooth$ to the dimension of a spline space $\splSpace^{\bsmoothr}$ obtained by relaxing the regularity requirements.
That is, for all interior edges $\edge$, it will be assumed that $\bsmoothr(\edge) \leq \bsmooth(\edge)$.
This relationship will be utilized to present sufficient conditions for \ginoun{} of the gBSS defined for $\splSpace^{\bsmoothr}$.

For the spline space $\splSpace^{\bsmoothr}$, let the first and last chain complexes in Equation \eqref{eq:complex} be denoted by $\idealComplex^{\bsmoothr}$ and $\quotientComplex^{\bsmoothr}$, respectively.
Then, by definition of the smoothness distributions $\bsmooth$ and $\bsmoothr$, we have the following inclusion map from $\idealComplex^{\bsmooth}$ to $\idealComplex^{\bsmoothr}$,
\begin{equation}
	\idealComplex^\bsmooth \rightarrow \idealComplex^{\bsmoothr}\;.
\end{equation}
Since both complexes are also included in the complex $\constantComplex$, we can build the following commuting diagram between two short exact sequences of chain complexes,
\begin{equation}
	\begin{tikzcd}
	0 \arrow[r] & \idealComplex^{\bsmooth} \arrow[r]\arrow[d] & \idealComplex^{\bsmoothr} \arrow[r] \arrow[d]& \idealComplex^{\bsmoothr} / \idealComplex^{\bsmooth} \arrow[r] \arrow[d] & 0\\
	0 \arrow[r] & \constantComplex \arrow[r] & \constantComplex \arrow[r] & 0 \arrow[r] & 0
	\end{tikzcd}
	\label{eq:double_complex_ideals}
\end{equation}

\begin{proposition}\label{prop:ideal_homology}
	If $H_0(\idealComplex^{\bsmoothr} / \idealComplex^{\bsmooth}) = 0$ and $\quotientComplex^\bsmooth$ is \gi{}, then the following hold,
	\begin{equation}
		H_0(\idealComplex^{\bsmoothr}) \isomorphic H_0(\constantComplex) \isomorphic H_0(\idealComplex^{\bsmooth})\;,
		\qquad
		H_0(\quotientComplex^{\bsmoothr}) = 0\;.
	\end{equation}
\end{proposition}
\begin{proof}
	The diagram in Equation \eqref{eq:double_complex_ideals} implies the following commuting diagram that connects the long exact sequence of homologies for the two exact sequences of complexes,
	\begin{equation}
	\begin{tikzcd}
	\cdots \arrow[r] & H_1(\idealComplex^\bsmoothr/\idealComplex^\bsmooth) \arrow[r] \arrow[d]
	& H_0(\idealComplex^\bsmooth) \arrow[r] \arrow[d]
	& H_0(\idealComplex^\bsmoothr) \arrow[r] \arrow[d]
	& H_0(\idealComplex^\bsmoothr/\idealComplex^\bsmooth) \arrow[r] \arrow[d]
	& 0\\
	\cdots \arrow[r] & 0 \arrow[r] 
	& H_0(\constantComplex) \arrow[r]
	& H_0(\constantComplex) \arrow[r] 
	& 0 \arrow[r]
	& 0
	\end{tikzcd}\;.
	\end{equation}
	Then, from Theorem \ref{thm:geom_indep_dim}, we know that $H_1(\quotientComplex^\bsmooth) \isomorphic H_0(\idealComplex^\bsmooth)$.
	Then, by an application of the Five lemma \cite{hatcher2002algebraic,schenck2003computational}, we obtain
	\begin{equation}
		H_0(\idealComplex^{\bsmoothr}) \isomorphic H_0(\constantComplex)\;,
	\end{equation}
	and the first part of the claim follows.
	The second part of the claim follows upon considering the long exact sequence implied by the following short exact sequence of chain complexes,
	\begin{equation}
		\begin{tikzcd}
			\idealComplex^\bsmoothr \arrow[r] &
			\constantComplex \arrow[r] & 
			\quotientComplex^\bsmoothr\;.
		\end{tikzcd}
	\end{equation}
\end{proof}

By focusing on the simpler object $H_0(\idealComplex^{\bsmoothr} / \idealComplex^{\bsmooth})$, see Lemma \ref{lem:support} at the end of this section, the previous result helps identify when $H_0(\idealComplex^{\bsmoothr})$ will be isomorphic to $H_0(\constantComplex)$.
Using this, we can now present our main results: sufficient conditions for the \ginoun{} of $\quotientComplex^\bsmoothr$.
The following results use the following commuting diagram of complexes,
\begin{equation}
\begin{tikzcd}
0 \arrow[r] & \idealComplex^{\bsmooth} \arrow[r]\arrow[d] & \constantComplex \arrow[r] \arrow[d]& \quotientComplex^{\bsmooth} \arrow[r] \arrow[d] & 0\\
0 \arrow[r] & \idealComplex^{\bsmoothr} \arrow[r] & \constantComplex \arrow[r] & \quotientComplex^{\bsmoothr} \arrow[r] & 0
\end{tikzcd}
\label{eq:double_complex_splines}
\end{equation}
which is built using the inclusion map $\constantComplex \rightarrow \constantComplex$.

\begin{proposition}\label{prop:quotient_homology}
	If $H_0(\idealComplex^{\bsmoothr} / \idealComplex^{\bsmooth}) = 0$ and $\quotientComplex^\bsmooth$ is \gi{}, then
	\begin{equation}
		H_1(\quotientComplex^{\bsmoothr}) = 0\;.
	\end{equation}
\end{proposition}
\begin{proof}
	The diagram in Equation \eqref{eq:double_complex_splines} implies the following commuting diagram that connects the long exact sequence of homologies for the two exact sequences of complexes,
	\begin{equation}
	\begin{tikzcd}
	\cdots \arrow[r]
	& H_1(\constantComplex) \arrow[r] \arrow[d]
	& H_1(\quotientComplex^\bsmooth) \arrow[r] \arrow[d]
	& H_0(\idealComplex^\bsmooth) \arrow[r] \arrow[d]
	& H_0(\constantComplex) \arrow[r] \arrow[d]
	& H_0(\quotientComplex^\bsmooth) \arrow[r] \arrow[d]
	& 0\\
	\cdots \arrow[r]
	& H_1(\constantComplex) \arrow[r]
	& H_1(\quotientComplex^\bsmoothr) \arrow[r]
	& H_0(\idealComplex^\bsmoothr) \arrow[r] 
	& H_0(\constantComplex) \arrow[r] 
	& H_0(\quotientComplex^\bsmoothr) \arrow[r] 
	& 0
	\end{tikzcd}\;.
	\end{equation}
	By an application of the Five lemma \cite{hatcher2002algebraic,schenck2003computational}, we obtain that the map $H_1(\quotientComplex^\bsmooth) \rightarrow H_1(\quotientComplex^\bsmoothr)$ must be a surjection.
	Then, the claim follows from the \ginoun{} of $\quotientComplex^\bsmooth$.
\end{proof}

\begin{corollary}\label{cor:new_dimension}
	If $H_0(\idealComplex^{\bsmoothr} / \idealComplex^{\bsmooth}) = 0$ and $\quotientComplex^\bsmooth$ is \gi{}, then $\quotientComplex^\bsmoothr$ is \gi{}.
\end{corollary}

As mentioned earlier, $H_0(\idealComplex^{\bsmoothr} / \idealComplex^{\bsmooth})$ is a simpler object to study, both computationally and analytically.
Let us precisely state what we mean by ``easier''.
Let ${\mesh}_1^\bsmoothr$ be the set of edges $\edge$ for which $\bsmoothr(\edge) < \bsmooth(\edge)$, and let ${\mesh}_0^\bsmoothr$ be the set of  vertices of the edges $\tau\in {\mesh}_1^\bsmoothr$ in $\meshInterior_0$.
Then, the following result follows: to study $H_0(\idealComplex^{\bsmoothr} / \idealComplex^{\bsmooth})$ we only need to focus on ${\mesh}_1^\bsmoothr$ and  ${\mesh}_0^\bsmoothr$.
\begin{lemma}\label{lem:support}
	The complex $\idealComplex^{\bsmoothr} / \idealComplex^{\bsmooth}$ is supported only on  ${\mesh}_1^\bsmoothr$ and  ${\mesh}_0^\bsmoothr$.
\end{lemma}
\begin{proof}
	The claim follows from the definition of the complexes $\idealComplex^{\bsmoothr}$ and $\idealComplex^{\bsmooth}$.
	Indeed, if $\bsmoothr(\edge) = \bsmooth(\edge)$, then $\ideal^{\bsmoothr}_\edge = \ideal^{\bsmooth}_\edge$ and the cokernel of the inclusion map from $\idealComplex^\bsmooth$ to $\idealComplex^\bsmoothr$ is zero on $\edge$; similarly for the vertices.
\end{proof}
	
	\section{Applications}\label{sec:examples}

Let us now see how we can compute the dimension of splines on interesting meshes using Corollary \ref{cor:new_dimension}.
This latter result is quite general and is applicable in a large number of settings.
We will narrow our focus down to the case where we reduce the smoothness across one or more interior edges $\edge$ from $\bsmooth(\edge)$ to $\bsmoothr(\edge) = -1$.  This is motivated as follows.
Let us say that we are working with splines on a polygonal partition of a topological disk $\Omega$.
Assume that Corollary \ref{cor:new_dimension} says that we can reduce the smoothness across any interior edge to $-1$.
Then, we can carve out arbitrary polygons from $\Omega$ to create new meshes of arbitrary topologies, and we can do so without losing the ability to exactly compute the dimension of splines on the new meshes.

\begin{definition}[Pruned mesh, spline space and gBSS complex]
	Given $\mesh$, $\bsmooth$ and $\mbf{\degreeu}$, let $F \subset \meshF$ be the set of faces such that, for any edge $\edge \subset \face \in F$, $\bsmooth(\edge) = -1$.
	\begin{itemize}
		\item A pruned mesh $\hat{\mesh}$ is obtained from $\mesh$ by deleting all faces in $F$.
		\item With $\hat{\bsmooth}$ and $\hat{\mbf{\degreeu}}$ defined by restricting $\bsmooth$ and $\mbf{\degreeu}$ to the edges and face of $\hat{\mesh}$, the spline space $\splSpace^{\hat{\bsmooth}}_{\hat{\mbf{\degreeu}}}(\hat{\mesh})$ will be called the pruned spline space.
		\item The pruned gBSS complex $\hat{\quotientComplex}^{\hat{\bsmooth}}$ is defined on $\hat{\mesh}$ using $\hat{\bsmooth}$ and $\hat{\mbf{\degreeu}}$.
		Equivalently, $\hat{\quotientComplex}^{\hat{\bsmooth}}$ can be obtained from ${\quotientComplex}^{{\bsmooth}}$ by deleting all faces $\face \in F$ from the top vector space.
	\end{itemize}
	
\end{definition}

Note that, in general, the domains corresponding to a mesh $\mesh$ and its corresponding pruned mesh $\hat{\mesh}$ will be topologically different.
The next three subsections show how the dimension of pruned spline spaces can be be computed for triangulations, polygonal meshes, and T-meshes.  We will use the following result in all sections.

\begin{theorem}\label{thm:hole_dimension}
	Let $\quotientComplex^\bsmooth$ be \gi{} for a given mesh $\mesh$, and let $\splSpace^{\hat{\bsmooth}}_{\hat{\mbf{\degreeu}}}(\hat{\mesh})$ and $\hat{\quotientComplex}^{\hat{\bsmooth}}$ be the corresponding pruned spline space and gBSS complex.
	Then,
	\begin{equation}
		\dimwp{\splSpace^{\hat{\bsmooth}}_{\hat{\mbf{\degreeu}}}(\hat{\mesh})}
		= \euler{\hat{\quotientComplex}^{\hat{\bsmooth}}}
		= \euler{\quotientComplex^\bsmooth} - \sum_{\face \in F} \dimwp{\PP_\face}\;.
	\end{equation}
\end{theorem}

\subsection{Triangulations}\label{ssec:triangulations}
Let us look at splines on given triangulations $\mesh$.
In the interest of a focused discussion, we will consider the following special case as our starting point,
\begin{equation}
\begin{split}
	&\PP_\face = \PP_\degreeu\;,\;\;\forall \face \in \meshF\;,\\
	&\bsmooth(\edge) \in \{\smooth, -1\}\;,\;\;\forall \edge \in \meshInteriorE\;,
\end{split}
\end{equation}
where $\degreeu \in \ZZP$, $\smooth \in \ZZ_{\geq -1}$, and $\PP_\degreeu$ is the space of polynomials of total degree at most $\degreeu$.

\begin{lemma}\label{lem:tri_edge}
	Let $\quotientComplex^\bsmooth$ be \gi{}, and consider an interior edge $\edge$ with end-points $\vertex$ and $\vertex'$ such that $\bsmooth(\edge) = \smooth$.
	If $\edge'$ is incident on $\vertex$ and $\bsmooth(\edge') = -1$, then $\quotientComplex^\bsmoothr$ is also \gi{}, where
	\begin{equation}
		\bsmoothr(\edge'') :=
		\begin{dcases}
			\bsmooth(\edge'')\;, & \edge'' \neq \edge\;,\\
			-1\;, & \text{otherwise}.
		\end{dcases}
	\end{equation}
\end{lemma}
\begin{proof}
	From Lemma \ref{lem:smoothness}, $\idealComplex^\bsmoothr/\idealComplex^\bsmooth$ is supported on
	\begin{itemize}
		\item $\edge$ and $\vertex'$: if $\bsmooth(\edge'') \neq -1$ for any edge $\edge''$ incident on $\vertex'$,
		\item $\edge$: otherwise.
	\end{itemize}
	Then, for both these cases, it is easy to verify that $H_0(\idealComplex^\bsmoothr/\idealComplex^\bsmooth)$ vanishes and the claim follows from Corollary \ref{cor:new_dimension}.
\end{proof}

Lemma~\ref{lem:tri_edge} tells us that if an edge $\edge$ is incident on a vertex $\vertex$ such that $\bsmooth(\edge) = -1$, then for any other edge $\edge'$ that is incident on $\vertex$, we can reduce the smoothness across $\edge'$ to $-1$.
On the other hand, we now consider what happens when we reduce the smoothness across the entire boundary of a face $\face \in \meshF$ to $-1$.

A dimension formula for splines on planar triangulations with holes is given in~\cite{alfeld_holes_1987} for $m\ge 4r+1$.  In the remarks at the end of~\cite{alfeld1990dimension} it is indicated how this dimension formula can be extended to $m\ge 3r+1$.  In this section we show that homological techniques (Theorem~\ref{thm:hole_dimension} and Corollary~\ref{cor:new_dimension}) can be used to derive the dimension formula for splines on triangulations with holes (and $m\ge 3r+1$) from the dimension formula on triangulations without holes.

The result depends on describing $H_0(\idealComplex^\bsmoothr/\idealComplex^\bsmooth)$, which is the cokernel of the only non-trivial map in the chain complex $\idealComplex^\bsmoothr/\idealComplex^\bsmooth$; we call this map $\phi$.  Suppose $\face$ is bounded by edges $\edge_1$, $\edge_2$ and $\edge_3$.  Let $\vertex_i = \edge_i \cap \edge_{i+1}$ be interior vertices of $\mesh$, where the index $i$ is cyclic in $(1, 2, 3)$.  Let $\ell_i$ be a linear form vanishing along $\edge_i$ for $i=1,2,3$.  Then $\phi$ is the map
\begin{equation}
\begin{tikzcd}[ampersand replacement=\&,column sep=huge]
	\phi~:~\begin{array}{c}
		\PP_\degreeu / \langle \ell_1^{\smooth+1} \rangle\\
		\moplus\\
		\PP_\degreeu / \langle \ell_2^{\smooth+1} \rangle\\
		\moplus\\
		\PP_\degreeu / \langle \ell_3^{\smooth+1} \rangle
	\end{array}\quad
	\arrow{r}{\begin{bmatrix}
		-1 & 1 & 0\\
		0 & -1 & 1\\
		1 & 0 & -1
	\end{bmatrix}} \&
	\quad
	\begin{array}{c}
	\PP_\degreeu / \ideal^\bsmooth_{\vertex_1}\\
	\moplus\\
	\PP_\degreeu / \ideal^\bsmooth_{\vertex_2}\\
	\moplus\\
	\PP_\degreeu / \ideal^\bsmooth_{\vertex_3}
	\end{array}\;.
\end{tikzcd}
\label{eq:tri_hole}
\end{equation}
Here
$\langle \ell_i^{r+1} \rangle$ is the subspace of $\PP_\degreeu$ containing all polynomial multiples of $\ell_i^{r+1}$ (for $i=1,2,3$).

\begin{lemma}\label{lem:tri_hole}
	Let $\quotientComplex^\bsmooth$ be \gi{} and consider a face $\face$ bounded by edges $\edge_1$, $\edge_2$ and $\edge_3$.
	Let $\vertex_i = \edge_i \cap \edge_{i+1}$ be interior vertices of $\mesh$, where the index $i$ is cyclic in $(1, 2, 3)$.
	Assume that the following hold for all $i \in \{1, 2, 3\}$,
	\begin{itemize}
		\item $\bsmooth(\edge_i) = \smooth$,
		\item $\bsmooth(\edge) \neq -1$ for any edge $\edge$ incident on $\vertex_i$,
	\end{itemize}
	If the cokernel of the map $\phi$ from Equation \eqref{eq:tri_hole} is trivial, then $\quotientComplex^\bsmoothr$ is \gi{}, where
	\begin{equation}
	\bsmoothr(\edge) :=
	\begin{dcases}
	\bsmooth(\edge)\;, & \edge \neq \edge_i\;, i = 1, 2, 3\;,\\
	-1\;, & \text{otherwise}.
	\end{dcases}
	\end{equation}
\end{lemma}
\begin{proof}
The is immediate from Corollary~\ref{cor:new_dimension} and the fact that $H_0(\idealComplex^\bsmoothr/\idealComplex^\bsmooth)$ is precisely the cokernel of the map $\phi$ from Equation \eqref{eq:tri_hole}.
\end{proof}

\begin{lemma}\label{lem:coker_description}
The cokernel of the map $\phi$ in Equation~\eqref{eq:tri_hole} (and hence $H_0(\idealComplex^\bsmoothr/\idealComplex^\bsmooth)$) is isomorphic to
\[
\PP_\degreeu / (\ideal^\bsmooth_{\vertex_1}+\ideal^\bsmooth_{\vertex_2}+\ideal^\bsmooth_{\vertex_3}).
\]
\end{lemma}
\begin{proof}
We apply the snake lemma to the left two columns of the following commutative diagram:
\[
\begin{tikzcd}[ampersand replacement=\&,column sep=huge]
0\arrow[d] \& 0\arrow[d] \\
\bigoplus\limits_{i=1}^3 \langle \ell_i^{\smooth+1} \rangle \arrow{r}{\phi^{''}}\arrow[d] \& \bigoplus\limits_{i=1}^3 \ideal^\bsmooth_{\vertex_i} \arrow{d}{\iota}\arrow[r] \& \mbox{coker}(\phi^{''})\arrow{d}{\overline{\iota}}\arrow[r]\& 0\\
\bigoplus\limits_{i=1}^3 \PP_\degreeu \arrow{r}{\phi^{'}=\begin{bmatrix}
	-1 & 1 & 0\\
	0 & -1 & 1\\
	1 & 0 & -1
	\end{bmatrix}} \arrow[d] \&\bigoplus\limits_{i=1}^3 \PP_\degreeu \arrow{d}{\pi}\arrow{r}{\begin{bmatrix}
	1 & 1 & 1
	\end{bmatrix}} \&  \PP_{\degreeu}\arrow{d}{\overline{\pi}}\arrow[r]\& 0 \\ 
\bigoplus\limits_{i=1}^3 \dfrac{\PP_\degreeu}{\langle \ell_i^{\smooth+1} \rangle} \arrow{r}{\phi}\arrow[d] \& \bigoplus\limits_{i=1}^3 \dfrac{\PP_\degreeu}{\ideal^\bsmooth_{\vertex_i}}\arrow[r]\arrow[d]\& \mbox{coker}(\phi)\arrow[d] \arrow[r]\& 0\\
0 \& 0 \& 0
\end{tikzcd}
\]
We only need the last portion of the snake lemma, namely the rightmost vertical column given by the sequence
\[
\mbox{coker}{\phi^{''}}\xrightarrow{\overline{\iota}} \mbox{coker}{\phi^{'}}\xrightarrow{\overline{\pi}} \mbox{coker}(\phi)\rightarrow 0,
\]
where $\overline{\iota}$ and $\overline{\pi}$ are the induced maps from $\iota$ and $\pi$.  In other words, $\mbox{coker}(\phi)\cong \mbox{coker}(\overline{\iota})$.  The above diagram explicitly identifies $\mbox{coker}{\phi^{'}}$ with $\PP_\degreeu$.  Since the diagram is commutative, $\overline{\iota}\cong \begin{bmatrix} 1 & 1 & 1\end{bmatrix}\circ\iota$.  The image of the latter inside $\PP_\degreeu$ is clearly $\ideal^\bsmooth_{\vertex_0}+\ideal^\bsmooth_{\vertex_1}+\ideal^\bsmooth_{\vertex_2}$.  Thus
\[
\mbox{coker}(\phi)\cong\mbox{coker}(\overline{\iota})\cong\PP_\degreeu / (\ideal^\bsmooth_{\vertex_0}+\ideal^\bsmooth_{\vertex_1}+\ideal^\bsmooth_{\vertex_2}),
\]
as claimed.
\end{proof}

\begin{remark}
	In the remainder of the manuscript, if $f_1,\ldots,f_k$ are polynomials in $\PP_\degreeu$ we use the notation $\langle f_i \rangle$ to denote the subspace of $\PP_\degreeu$ containing all polynomial multiples of $f_i$ and $\langle f_1,\ldots,f_k\rangle$ to denote $\sum \langle f_i\rangle$.
\end{remark}

By a change of coordinates we may assume that $\ell_1=x,\ell_2=y,$ and $\ell_3=z$, where $z=(x+y+1)$.  Then 
\[
\begin{array}{rl}
\ideal^\bsmooth_{\vertex_1}= & \langle x^{r+1},y^{r+1},L_1^{r+1},L_2^{r+1},\cdots,L_{t_1-2}^{r+1} \rangle\\
\ideal^\bsmooth_{\vertex_2}= & \langle y^{r+1},z^{r+1},M_1^{r+1},M_2^{r+1},\cdots,M_{t_2-2}^{r+1} \rangle\\
\ideal^\bsmooth_{\vertex_3}= & \langle x^{r+1},z^{r+1},N_1^{r+1},N_2^{r+1},\cdots,N_{t_3-2}^{r+1} \rangle,
\end{array}
\]
where $L_i,M_i,N_i$ are linear forms in $x$ and $y$, $y$ and $z$, $x$ and $z$, respectively (regarding $z$ as a variable).  The integer $t_i$ is the maximum number of the powers of linear forms $\{\ell_\edge^{r+1}: \vertex_i\in \tau\}$ which are linearly independent (this is the minimum of $r+2$ and the number of slopes incident upon $\vertex_i$).  A standard basis for $\PP_m$ is provided by the monomials $x^iy^j$, $i+j\le m$.  An alternative basis for $\PP_m$ which is more convenient for our arguments is the polynomials of the form $x^iy^jz^k=x^iy^j(x+y+1)^k$, where $i+j+k=d$.  Lemma~\ref{lem:coker_description} guarantees that $H_0(\idealComplex^\bsmoothr/\idealComplex^\bsmooth)$ will vanish if every polynomial of this form is in the sum $\ideal^\bsmooth_{\vertex_1}+\ideal^{\bsmooth}_{\vertex_2}+\ideal^\bsmooth_{\vertex_3}$.  We can obtain good estimates for this from the integers
\[
\begin{array}{rl}
\Omega_1:= & \min\{d:x^iy^j\in \ideal^\bsmooth_{\vertex_1} \mbox{ for all } i+j\ge d \}\\
\Omega_2:= & \min\{d:y^iz^j\in \ideal^\bsmooth_{\vertex_2} \mbox{ for all } i+j\ge d \}\\
\Omega_3:= & \min\{d:x^iz^j\in \ideal^\bsmooth_{\vertex_3} \mbox{ for all } i+j\ge d \}.
\end{array}
\]
We can in fact obtain these exactly using~\cite{schenck1997family}.
\begin{lemma}\cite[Corollary~3.4]{schenck1997family}\label{lem:regularity}
Let the vertices $\vertex_1,\vertex_2,\vertex_3$ be as above.  Suppose $\vertex_i$ has $n_i$ distinct slopes and put $t_i=\min\{r+2,n_i\}$ for $i=1,2,3$.  Suppose further that $\Omega_i$ is defined as above.  Then, for $i=1,2,3$,
\[
\Omega_i=r+\left\lceil\frac{r+1}{t_i-1}\right\rceil.
\]
\end{lemma}

We are now in a position to state our main result on triangulations.

\begin{theorem}\label{thm:triangleRemoval}
Let $\quotientComplex^\bsmooth$ be \gi{} and consider a face $\face$ bounded by edges $\edge_1$, $\edge_2$ and $\edge_3$.  Let $\vertex_i = \edge_i \cap \edge_{i+1}$ be interior vertices of $\mesh$, where the index $i$ is cyclic in $(1, 2, 3)$. Assume that the following hold for all $i \in \{1, 2, 3\}$,
\begin{itemize}
	\item $\bsmooth(\edge_i) = \smooth$,
	\item $\bsmooth(\edge) \neq -1$ for any edge $\edge$ incident on $\vertex_i$,
\end{itemize}
Suppose $\vertex_i$ has $n_i$ distinct slopes, put $t_i=\min\{r+1,n_i\}$ and $\Omega_i=r+\lceil\frac{r+1}{t_i-1}\rceil$ for $i=1,2,3$.  Then, if
\[
m>\frac{\Omega_1+\Omega_2+\Omega_3-3}{2}
\]
the chain complex $\quotientComplex^\bsmoothr$ for the pruned mesh is \gi{}, where
\begin{equation}
\bsmoothr(\edge) :=
\begin{dcases}
\bsmooth(\edge)\;, & \edge \neq \edge_i\;, i = 1, 2, 3\;,\\
-1\;, & \text{otherwise}.
\end{dcases}
\end{equation}
\end{theorem}
\begin{proof}
By Lemma~\ref{lem:tri_hole}, it suffices to show that the cokernel of the map $\phi$ in Equation~\ref{eq:tri_hole} is trivial.  By Lemma~\ref{lem:coker_description} it suffices to show that $\ideal^\bsmooth_{\vertex_1}+\ideal^\bsmooth_{\vertex_1}+\ideal^\bsmooth_{\vertex_2}=\PP_m$ for $m> \frac{\Omega_1+\Omega_2+\Omega_3-3}{2}$.  Suppose that $x^iy^jz^k\notin \ideal^\bsmooth_{\vertex_1}+\ideal^\bsmooth_{\vertex_1}+\ideal^\bsmooth_{\vertex_2}$.  Then, by Lemma~\ref{lem:regularity}, we must have
\[
\begin{array}{c}
i+j\le \Omega_1-1\\
j+k\le \Omega_2-1\\
i+k\le \Omega_3-1.
\end{array}
\]
Summing these leads to $2(i+j+k)\le\Omega_1+\Omega_2+\Omega_3-3$, or $m=(i+j+k)\le\frac{\Omega_1+\Omega_2+\Omega_3-3}{2}$.
\end{proof}


\begin{theorem}\label{thm:3r+1}
Suppose $\mesh$ is any triangulation.  Then $\quotientComplex^\bsmooth$ is \gi{} for $m\ge 3r+1$.  In particular, the formula in~\cite[Equation~7.1]{alfeld1990dimension} holds for $m\ge 3r+1$.
\end{theorem}
\begin{proof}
An arbitrary planar triangulation $\mesh$ has some number of holes, along with a minimal number of triangles needed to `fill in' those holes with a triangulation.  We induct on this minimal number of triangles needed to `fill in' the holes.  If no such triangles are needed, then $\mesh$ triangulates a simply connected region and the main result of~\cite{alfeld1990dimension} implies that the corresponding chain complex $\quotientComplex^\bsmooth$ is \gi{} for $m\ge 3r+1$.  Now suppose $\mesh$ triangulates an arbitrary non-simply connected region with corresponding chain complex $\quotientComplex^\bsmoothr$.  Pick one of the holes in $\mesh$ and form $\mesh'$ by adding in a triangular face $\sigma$ so that
\begin{enumerate}
\item $\sigma$ begins a filling of the holes of $\mesh$ in a minimal fashion
\item $\sigma\cap\mesh$ is connected
\end{enumerate}
By induction, the chain complex $\quotientComplex^\bsmooth$ corresponding to $\mesh'$ is \gi{} for $m\ge 3r+1$.  If $\sigma\cap\mesh$ is not the entire boundary of $\sigma$, then applying Lemma~\ref{lem:tri_edge} at most twice yields that $\quotientComplex^\bsmoothr$ is \gi{} for $m\ge 3r+1$.  If $\sigma\cap\mesh$ is the entire boundary of $\sigma$, then applying Theorem~\ref{thm:triangleRemoval} yields that $\quotientComplex^\bsmoothr$ is \gi{} for $m\ge 3r+1$ as long as $(\Omega_1+\Omega_2+\Omega_3-3)/2< 3r+1$.  The integers $\Omega_i$ are largest when only two slopes meet at $\vertex_i$.  In this case $\Omega_i=2r+1$.  Hence
\[
(\Omega_1+\Omega_2+\Omega_3-3)/2\le (6r)/2= 3r<3r+1,
\]
completing the induction.
\end{proof}

\begin{remark}
In essence, Theorem~\ref{thm:3r+1} recovers~\cite[Equation~7.1]{alfeld1990dimension} from~\cite[Theorem~2.2]{alfeld1990dimension}.
\end{remark}


Lemma~\ref{lem:coker_description} can be used to give better estimates in non-generic situations, as we illustrate in the following corollary.
\begin{corollary}\label{cor:tri_special_cases}
Let $\smooth$, $\degreeu$ and $\mesh$ be such that $\quotientComplex^\bsmooth$ is \gi{}.  If every interior vertex of $\mesh$ has at least $r+2$ distinct slopes incident upon it and $\degreeu>\frac{3r}{2}$ then the smoothness across the boundary of any face $\face\in\meshF$ can be reduced to $-1$ while preserving \ginoun{}.  In particular,
	\begin{enumerate}[label=(\Alph*)]
		\item $\degreeu \geq 2$ and $\smooth=1$:
		If $\mesh$ is such that any interior vertex has edges with at least $3$ distinct slopes incident upon it,
		then the smoothness across the boundary of any face $\face \in \meshF$ can be reduced to $-1$ while preserving \ginoun{}.
		
		\item $\degreeu \geq 4$ and $\smooth=2$:
		If $\mesh$ is such that any interior vertex has edges with at least $4$ distinct slopes incident upon it,
		then the smoothness across the boundary of any face $\face \in \meshF$ can be reduced to $-1$ while preserving \ginoun{}.
	\end{enumerate}
	In both of the above cases, the dimension of the corresponding pruned spline space can be computed by a direct application of Theorem \ref{thm:hole_dimension}.
\end{corollary}

\begin{example}[$C^1$ quadratic splines on a domain with holes]\label{ex:triangulations}
	Consider the meshes shown in Figure \ref{fig:tri_2holes}.
	We are interested in the space of $C^1$ quadratic splines on mesh in (a).
	Then, we can interpret this space as the pruned version of the space of $C^1$ quadratics on on the triangulation in (b) after we have reduced the smoothness along the dashed edges in (c) to $-1$.
	Therefore, we start by looking at the mesh in (b).
	
	First, for the mesh in (b), choose all $\PP_\face = \PP_2$ for all faces $\face$, where $\PP_2$ is the space of polynomials of total degree at most $2$, and $\bsmooth(\edge) = 1$ for all interior edges $\edge$.
	It can be checked, using the formulas in \cite{mourrain2013homological}, for instance, that $\quotientComplex^\bsmooth$ is \gi{} and the dimension of the corresponding space is $27$.
	
	Then, using Lemmas \ref{lem:tri_edge} and \ref{lem:tri_hole}, we can reduce the smoothness across all dashed edges in (c) to $-1$ while preserving \ginoun{} (c.f. Corollary \ref{cor:tri_special_cases}(A)).
	The dimension of the resulting space is $53$.
	As a result, the dimension of the pruned space in (a) can be exactly computed using Theorem \ref{thm:hole_dimension} to be $29$.
	
	Note that, for the mesh in (a), $H_1(\constantComplex)$ is not $0$.
	Therefore, it is not directly covered by the approach presented in \cite{mourrain2013homological} for dimension counting.
	The result of the computation of course coincides with Billera \cite[Theorem 5.8]{billera1988homology} since we looked at the special choice of $\smooth=1$.
	Nonetheless, our results can also be applied for different choices of $\bsmooth$; c.f. Corollary \ref{cor:tri_special_cases}(B).
\end{example}

\begin{rem}\label{rem:dim_increase}
	At first glance, it may seem strange that the dimension of splines in Figure \ref{fig:tri_2holes}(b) is smaller than the dimension of splines in (a).
	However, this makes sense because, while removing some faces from the mesh, we also removed the smoothness constraints across the boundaries of those faces.
	The net effect of such operations may very well lead to an increase in the dimension, as is the case here.
	The same observation will also hold later when we look at T-meshes in Example \ref{ex:tmesh}.
\end{rem}

\begin{figure}[t]
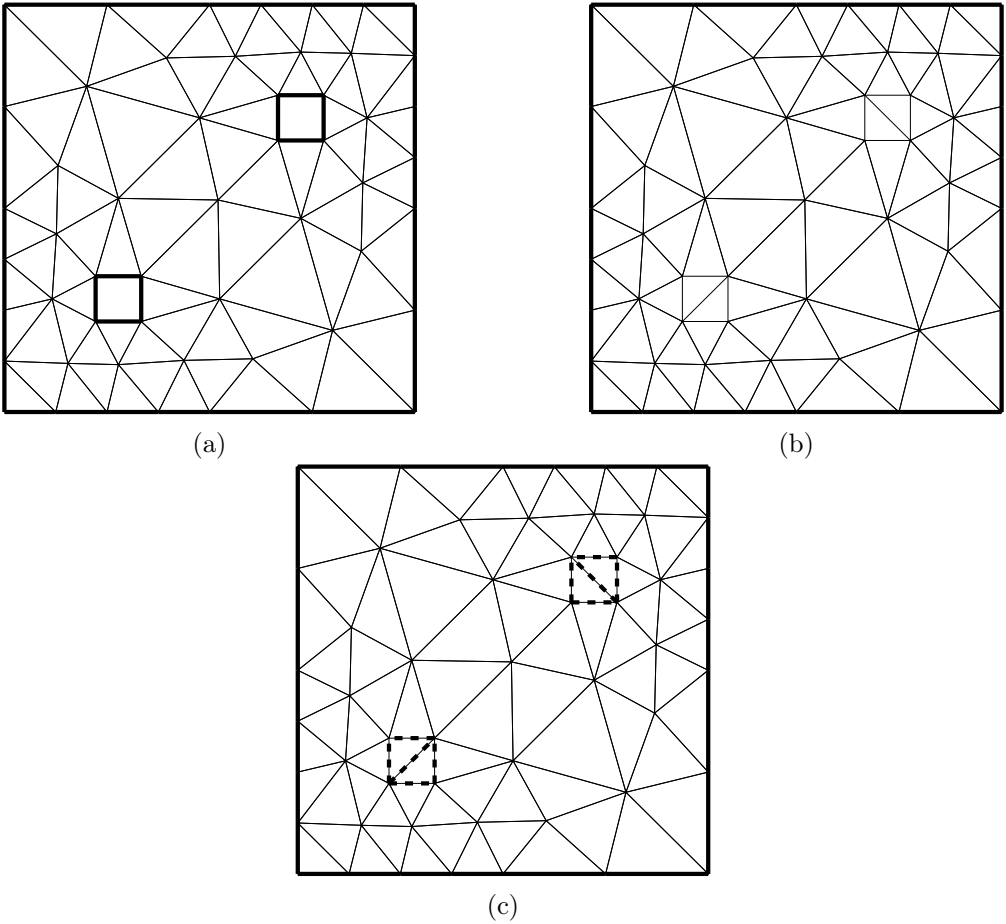

	\centering
	\subcaptionbox{}[0.49\textwidth]{\includetikz{tri_2holes}{./tikz}}
	\subcaptionbox{}[0.49\textwidth]{\includetikz{tri_2holes_1}{./tikz}}\\
	\subcaptionbox{}[0.49\textwidth]{\includetikz{tri_2holes_0}{./tikz}}\\
	\caption{A non-simply connected domain and its triangulation are shown in (a).
		Starting from the triangulation in (b), we can interpret the triangulation in (a) as a pruned triangulation once the smoothness across the dashed edges in (c) have been reduced to $-1$.
		In (c), reducing the smoothness to $-1$ decouples the faces inside the dashed region from those outside the dashed region.
		Then, the required spine space dimension can be computed from the one on (c) using Theorem \ref{thm:hole_dimension}, i.e., by subtracting $\dimwp{\PP_\face}$ for each $\face$ contained inside the dashed region.
		Note that all domain boundaries have been displayed in bold.}
	\label{fig:tri_2holes}
\end{figure}

\subsection{Polygonal meshes}
Now suppose $\mesh$ is an arbitrary rectilinear mesh (allowing polygonal faces) with the same setup as in Section~\ref{ssec:triangulations}.  That is,
\begin{equation}
\begin{split}
&\PP_\face = \PP_\degreeu\;,\;\;\forall \face \in \meshF\;,\\
&\bsmooth(\edge) \in \{\smooth, -1\}\;,\;\;\forall \edge \in \meshInteriorE\;,
\end{split}
\end{equation}
where $\degreeu \in \ZZP$, $\smooth \in \ZZ_{\geq -1}$, and $\PP_\degreeu$ is the space of polynomials of total degree at most $\degreeu$.  For technical reasons we will assume that the polygonal faces are convex although this condition could be dropped for particular examples.

We consider reducing smoothness to $-1$ across the boundary of a face $\face\in\mesh$.  Suppose $\face$ is bounded by edges $\edge_1$, $\edge_2,\ldots,\edge_k$.  Let $\vertex_i = \edge_i \cap \edge_{i+1}$ be interior vertices of $\mesh$, where the index $i$ is cyclic in $(1, 2, \ldots, k)$.  Let $\ell_i$ be a linear form vanishing along $\edge_i$ for $i=1,2,\ldots,k$.  Let $\bsmoothr$ be the smoothness distribution where
\begin{equation}
\bsmoothr(\edge) :=
\begin{dcases}
\bsmooth(\edge)\;, & \edge \neq \edge_i\;, i = 1, \ldots, k\;,\\
-1\;, & \text{otherwise}.
\end{dcases}
\end{equation}

Then the non-trivial map $\phi$ in $H_0(\idealComplex^\bsmoothr/\idealComplex^\bsmooth)$ is given by

\begin{equation}
\begin{tikzcd}[ampersand replacement=\&,column sep=huge]
\phi~:~\begin{array}{c}
\PP_\degreeu / \langle \ell_1^{\smooth+1} \rangle\\
\moplus\\
\PP_\degreeu / \langle \ell_2^{\smooth+1} \rangle\\
\moplus\\
\vdots\\
\moplus\\
\PP_\degreeu / \langle \ell_k^{\smooth+1} \rangle
\end{array}\quad
\arrow{r}{\begin{bmatrix}
	-1 & 1 & \cdots & 0\\
	0 & -1 & \cdots & 0 \\
	\vdots & \vdots & \ddots & \vdots\\
	1 & 0 & \cdots & -1
	\end{bmatrix}} \&
\quad
\begin{array}{c}
\PP_\degreeu / \ideal^\bsmooth_{\vertex_1}\\
\moplus\\
\PP_\degreeu / \ideal^\bsmooth_{\vertex_2}\\
\moplus\\
\vdots\\
\moplus\\
\PP_\degreeu / \ideal^\bsmooth_{\vertex_k}
\end{array}\;.
\end{tikzcd}
\label{eq:poly_hole}
\end{equation}

\begin{lemma}\label{lem:polygonalcokernel}
The cokernel of $\phi$ (and hence $H_0(\idealComplex^\bsmoothr/\idealComplex^\bsmooth)$) is isomorphic to $\PP_\degreeu/(\ideal^\bsmooth_{\vertex_1}+\cdots+\ideal^\bsmooth_{\vertex_k})$.
\end{lemma}
\begin{proof}
The proof is the same as Lemma~\ref{lem:coker_description}.
\end{proof}

If the face $\face$ is not triangular, then it is likely that $(\ideal^\bsmooth_{\vertex_1}+\cdots+\ideal^\bsmooth_{\vertex_k})=\PP_m$ for quite small $m$ relative to $r$ and it is possible to obtain quite accurate estimates for the smallest such $m$.  However, this equation typically holds in degree far lower than dimension formulas are actually known (see~\cite{dipasquale2018dimension}), so we focus on giving some coarse estimates that are easy to derive.

\begin{proposition}\label{prop:polygonal}
Let $\mesh$ be a planar polygonal mesh and $\face$ a face of $\mesh$ with bounding edges $\edge_1,\ldots,\edge_k$ and vertices $\vertex_i=\edge_i\cap\edge_{i+1}$ ($i$ taken cyclically from $(1,\ldots,k)$).  Let $\bsmoothr,\bsmooth$ be the smoothness distributions given above, where $\bsmoothr$ reduces smoothness along the edges of $\face$ to $-1$.  As above, let $\Omega_i=r+\lceil \frac{r+1}{t_i-1}\rceil$, where $t_i=\min\{r+1,n_i\}$ and $n_i$ is the number of slopes incident at $\vertex_i$.  Then $H_0(\idealComplex^\bsmoothr/\idealComplex^\bsmooth)$ vanishes for
\begin{itemize}
\item $m>3r$ or
\item $m>\Omega_i+\Omega_{i+1}-2$ for any $i=1,\ldots,k$.
\end{itemize}
In particular, if $\quotientComplex^\bsmooth$ is \gi{} then so is $\quotientComplex^\bsmoothr$.
\end{proposition}

\begin{proof}
By Lemma~\ref{lem:polygonalcokernel} it suffices to show that $(\ideal^\bsmooth_{\vertex_1}+\cdots+\ideal^\bsmooth_{\vertex_k})=\PP_m$ for the two cases above.  Without loss of generality we can change coordinates so that $\ideal^\bsmooth_{\edge_1}=\langle x^{r+1} \rangle, \ideal^\bsmooth_{\edge_2}=\langle y^{r+1}\rangle ,$ and $\ideal^\bsmooth_{\edge_i}=\langle z^{r+1}\rangle$ for some $1<i<k$, where $z=(x+y+1)$.  Then it is clear that $\langle x^{r+1},y^{r+1},z^{r+1}\rangle \subset (\ideal^\bsmooth_{\vertex_1}+\cdots+\ideal^\bsmooth_{\vertex_k})$.  We again choose to use the basis $x^iy^jz^k$, $i+j+k=\degreeu$, for $\PP_\degreeu$.  If $x^iy^jz^k\notin \langle x^{r+1},y^{r+1},z^{r+1}\rangle$ then we must have $i\le r, j\le r,$ and $k\le r$; thus $i+j+k\le 3r$.  It follows that $\langle x^{r+1},y^{r+1},z^{r+1}\rangle=\PP_\degreeu$, hence also $(\ideal^\bsmooth_{\vertex_1}+\cdots+\ideal^\bsmooth_{\vertex_k})=\PP_\degreeu$.

Now suppose that $m>\Omega_i+\Omega_{i+1}-2$ for some $i=1,\ldots,k$.  Again, changing coordinates we may assume that
\[
\begin{array}{rl}
\ideal^\bsmooth_{\vertex_i}= & \langle x^{r+1},y^{r+1},L_1^{r+1},L_2^{r+1},\cdots,L_{t_1-2}^{r+1} \rangle\\
\ideal^\bsmooth_{\vertex_{i+1}}= & \langle y^{r+1},z^{r+1},M_1^{r+1},M_2^{r+1},\cdots,M_{t_2-2}^{r+1} \rangle\\
\end{array}
\]
where $L_i$ and $M_i$ are linear forms in $x$ and $y$, and $y$ and $z$, respectively (regarding $z$ as a variable, where we are again letting $z=x+y-1$). 
Recall that $x^iy^jz^k$, $i+j+k=m$, form a basis for $\PP_{\degreeu}$. 
Suppose $x^iy^jz^k\notin \ideal^\bsmooth_{\vertex_i}+\ideal^\bsmooth_{\vertex_{i+1}}$.  Then $i+j\le \Omega_i-1$ and $j+k\le\Omega_{i+1}-1$.  Hence $i+j+k\le i+2j+k\le \Omega_i+\Omega_{i+1}-2$, contrary to assumption.  Hence $\ideal^\bsmooth_{\vertex_i}+\ideal^\bsmooth_{\vertex_{i+1}}=\PP_{\degreeu}$, and thus $(\ideal^\bsmooth_{\vertex_1}+\cdots+\ideal^\bsmooth_{\vertex_k})=\PP_{\degreeu}$ as well.
\end{proof}

\begin{remark}
Suppose $\mesh$ is a polygonal mesh with smoothness distribution $\bsmooth$.  Let $\bsmoothr$ be the smoothness distribution which is equal to $\bsmooth$ on every edge other than $\edge$, and satisfies $\bsmoothr(\edge)=-1$.  Suppose $\edge$ joins vertices $\vertex_1$ and $\vertex_2$.  Let $\Omega_i=r+\lceil \frac{r+1}{t_i-1}\rceil$, where $t_i=\min\{r+1,n_i\}$ and $n_i$ is the number of slopes incident at $\vertex_i$ ($i=1,2$).  If $\degreeu>\Omega_1+\Omega_2-2$, the proof of Proposition~\ref{prop:polygonal} shows that if $\quotientComplex^\bsmooth$ is lower acyclic then so is $\quotientComplex^\bsmoothr$.  This observation can be used to remove all smoothness requirements along arbitrary edges, as long as $\degreeu$ is large enough.
\end{remark}

\begin{example}[$C^1$ splines on a polygonal mesh with holes] Let $\mesh$ be the mesh depicted in Figure~\ref{fig:truncated_cube_schlegel}~(b).  For simplicity we assume the coordinates of the vertices in this figure are chosen generically (thus $\quotientComplex^\bsmooth$ is \gi{} for large $m$).  In this case~\cite{mcdonald_schenck_09} implies that $\dim \splSpace^1=\binom{m+2}{2}-20\binom{m}{2}+32\binom{m-1}{2}$ for $m\gg 0$.  The main result of~\cite{dipasquale2018dimension} implies that $\dim \splSpace^1=\binom{m+2}{2}-20\binom{m}{2}+32\binom{m-1}{2}$ for $m\ge 26$.  However, computations indicate that $\quotientComplex^\bsmooth$ is \gi{} for $m\ge 7$.  Trusting these computations,  Proposition~\ref{prop:polygonal} indicates that the pruned spline space over the mesh depicted in Figure~\ref{fig:truncated_cube_schlegel} (a) will satisfy the dimension formula $-4\binom{m}{2}+16\binom{m-1}{2}$ for $m\ge 7$.  (In fact, computations indicate that the pruned spline space satisfies this formula for $m\ge 4$.)
\end{example}

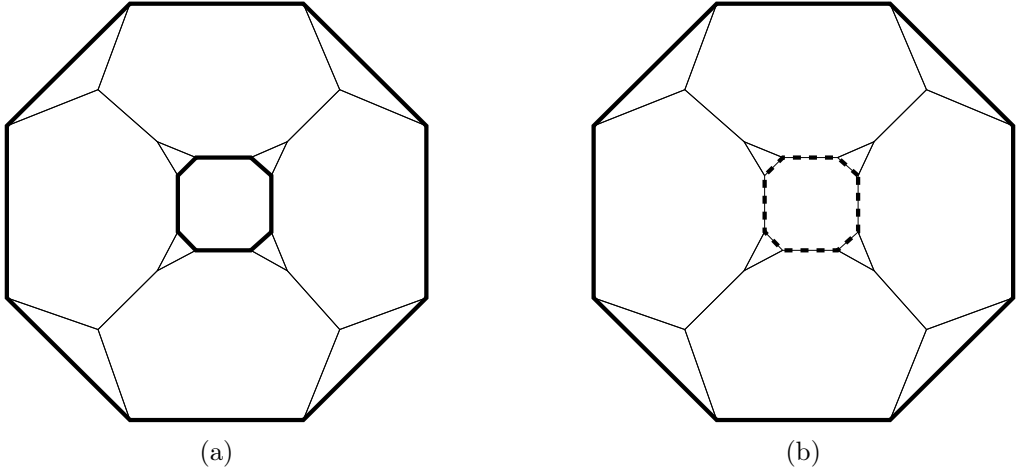
\begin{figure}[t]
	\centering
	\subcaptionbox{}[0.49\textwidth]{\begin{tikzpicture}[scale=3]
\tikzset{
	bThickness/.style={line width=#1\pgflinewidth},
	bThickness/.default={2},
}

\tikzset{
	eThickness/.style={line width=#1\pgflinewidth},
	eThickness/.default={0.5},
}

\begin{scope}
\coordinate (v0) at (0.92, 0.38) {};
\coordinate (v1) at (0.38, 0.92) {};
\coordinate (v2) at (-0.38, 0.92) {};
\coordinate (v3) at (-0.92, 0.38) {}; 
\coordinate (v4) at (-0.92, -0.38) {};
\coordinate (v5) at (-0.38, -0.92) {}; 
\coordinate (v6) at (0.38, -0.92) {};
\coordinate (v7) at (0.92, -0.38) {};
\coordinate (v8) at (0.54, 0.54) {};
\coordinate (v9) at (-0.52, 0.54) {};
\coordinate (v10) at (-0.52, -0.52) {};
\coordinate (v11) at (0.54, -0.52) {}; 
\coordinate (v12) at (0.31, 0.31) {}; 
\coordinate (v13) at (-0.26, 0.31) {}; 
\coordinate (v14) at (-0.26, -0.26) {};
\coordinate (v15) at (0.31, -0.26) {};
\coordinate (v16) at (0.24, 0.16) {}; 
\coordinate (v17) at (0.15, 0.24) {};
\coordinate (v18) at (-0.090, 0.24) {};
\coordinate (v19) at (-0.17, 0.16) {}; 
\coordinate (v20) at (-0.17, -0.088) {};
\coordinate (v21) at (-0.091, -0.17) {};
\coordinate (v22) at (0.15, -0.17) {};
\coordinate (v23) at (0.24, -0.089) {};

\draw[eThickness] (v0.center) -- (v1.center) -- (v8.center) -- cycle;
\draw[eThickness] (v2.center) -- (v3.center) -- (v9.center) -- cycle;
\draw[eThickness] (v4.center) -- (v5.center) -- (v10.center) -- cycle;
\draw[eThickness] (v6.center) -- (v7.center) -- (v11.center) -- cycle;
\draw[eThickness] (v15.center) -- (v22.center) -- (v23.center) -- cycle;
\draw[eThickness] (v12.center) -- (v16.center) -- (v17.center) -- cycle;
\draw[eThickness] (v13.center) -- (v18.center) -- (v19.center) -- cycle;
\draw[eThickness] (v14.center) -- (v20.center) -- (v21.center) -- cycle;
\draw[eThickness] (v0.center) -- (v8.center) -- (v12.center) -- (v16.center) -- (v23.center) -- (v15.center) -- (v11.center) -- (v7.center) -- cycle;
\draw[eThickness] (v1.center) -- (v2.center) -- (v9.center) -- (v13.center) -- (v18.center) -- (v17.center) -- (v12.center) -- (v8.center) -- cycle;
\draw[eThickness] (v3.center) -- (v4.center) -- (v10.center) -- (v14.center) -- (v20.center) -- (v19.center) -- (v13.center) -- (v9.center) -- cycle;
\draw[eThickness] (v5.center) -- (v6.center) -- (v11.center) -- (v15.center) -- (v22.center) -- (v21.center) -- (v14.center) -- (v10.center) -- cycle;

\draw[bThickness] (v16.center) -- (v17.center) -- (v18.center) -- (v19.center) -- (v20.center) -- (v21.center) -- (v22.center) -- (v23.center) -- cycle;
\draw[bThickness] (v0.center) -- (v1.center) -- (v2.center) -- (v3.center) -- (v4.center) -- (v5.center) -- (v6.center) -- (v7.center) -- cycle;
\end{scope}
\end{tikzpicture}}
	\subcaptionbox{}[0.49\textwidth]{\begin{tikzpicture}[scale=3]

\tikzset{
	bThickness/.style={line width=#1\pgflinewidth},
	bThickness/.default={2},
}

\tikzset{
	eThickness/.style={line width=#1\pgflinewidth},
	eThickness/.default={0.5},
}

\begin{scope}
\coordinate (v0) at (0.92, 0.38) {};
\coordinate (v1) at (0.38, 0.92) {};
\coordinate (v2) at (-0.38, 0.92) {};
\coordinate (v3) at (-0.92, 0.38) {}; 
\coordinate (v4) at (-0.92, -0.38) {};
\coordinate (v5) at (-0.38, -0.92) {}; 
\coordinate (v6) at (0.38, -0.92) {};
\coordinate (v7) at (0.92, -0.38) {};
\coordinate (v8) at (0.54, 0.54) {};
\coordinate (v9) at (-0.52, 0.54) {};
\coordinate (v10) at (-0.52, -0.52) {};
\coordinate (v11) at (0.54, -0.52) {}; 
\coordinate (v12) at (0.31, 0.31) {}; 
\coordinate (v13) at (-0.26, 0.31) {}; 
\coordinate (v14) at (-0.26, -0.26) {};
\coordinate (v15) at (0.31, -0.26) {};
\coordinate (v16) at (0.24, 0.16) {}; 
\coordinate (v17) at (0.15, 0.24) {};
\coordinate (v18) at (-0.090, 0.24) {};
\coordinate (v19) at (-0.17, 0.16) {}; 
\coordinate (v20) at (-0.17, -0.088) {};
\coordinate (v21) at (-0.091, -0.17) {};
\coordinate (v22) at (0.15, -0.17) {};
\coordinate (v23) at (0.24, -0.089) {};

\draw[eThickness] (v0.center) -- (v1.center) -- (v8.center) -- cycle;
\draw[eThickness] (v2.center) -- (v3.center) -- (v9.center) -- cycle;
\draw[eThickness] (v4.center) -- (v5.center) -- (v10.center) -- cycle;
\draw[eThickness] (v6.center) -- (v7.center) -- (v11.center) -- cycle;
\draw[eThickness] (v15.center) -- (v22.center) -- (v23.center) -- cycle;
\draw[eThickness] (v12.center) -- (v16.center) -- (v17.center) -- cycle;
\draw[eThickness] (v13.center) -- (v18.center) -- (v19.center) -- cycle;
\draw[eThickness] (v14.center) -- (v20.center) -- (v21.center) -- cycle;
\draw[eThickness] (v0.center) -- (v8.center) -- (v12.center) -- (v16.center) -- (v23.center) -- (v15.center) -- (v11.center) -- (v7.center) -- cycle;
\draw[eThickness] (v1.center) -- (v2.center) -- (v9.center) -- (v13.center) -- (v18.center) -- (v17.center) -- (v12.center) -- (v8.center) -- cycle;
\draw[eThickness] (v3.center) -- (v4.center) -- (v10.center) -- (v14.center) -- (v20.center) -- (v19.center) -- (v13.center) -- (v9.center) -- cycle;
\draw[eThickness] (v5.center) -- (v6.center) -- (v11.center) -- (v15.center) -- (v22.center) -- (v21.center) -- (v14.center) -- (v10.center) -- cycle;

\draw[bThickness, dashed] (v16.center) -- (v17.center) -- (v18.center) -- (v19.center) -- (v20.center) -- (v21.center) -- (v22.center) -- (v23.center) -- cycle;
\draw[bThickness] (v0.center) -- (v1.center) -- (v2.center) -- (v3.center) -- (v4.center) -- (v5.center) -- (v6.center) -- (v7.center) -- cycle;
\end{scope}
\end{tikzpicture}}\\
	\caption{A non-simply connected polygonal mesh is shown in (a).
		We can interpret the mesh in (a) as a pruned mesh once the smoothness across the dashed edges in (b) have been reduced to $-1$.
		In (b), reducing the smoothness to $-1$ decouples the faces inside the dashed region from those outside the dashed region.
		Then, the required spine space dimension can be computed from the one on (a) using Theorem \ref{thm:hole_dimension}, i.e., by subtracting the term $\dimwp{\PP_\face}$ for the polygon enclosed by the dashed edges.
		All domain boundaries have been displayed in bold.}
	\label{fig:truncated_cube_schlegel}
\end{figure}

\subsection{T-meshes}

Let us now present examples of applications to splines on T-meshes $\mesh$.
In particular, we will show how Corollary \ref{cor:new_dimension} can be combined with previously published results from \cite{mourrain2014dimension,toshniwal_polynomial_2019,toshniwal_mixed_2019} to compute the dimension of bi-degree splines in a very general setting by reducing the smoothness across one or more edges to $C^{-1}$.
Thereafter, Theorem \ref{thm:hole_dimension} will allow us to compute the dimension of the corresponding pruned spline spaces on T-meshes of arbitrary topologies.

T-meshes have a simpler structure 
and as a result we can consider a more general setting than the one we discussed in the previous sub-sections.
More precisely, for $\face \in \meshF$, we will allow $\mbf{\degreeu}(\face) = \PP_\face$ to be the vector space of polynomials of bi-degree at most $(\degreeu_\face, \degreeu_\face)$ for some $\degreeu_\face \in \ZZP$, i.e.,
\begin{equation}
	\mbf{\degreeu}(\face) = \PP_\face = \PP_{\degreeu_\face\degreeu_\face}\;.
\end{equation}
Of course, we will assume that the assumptions placed on $\mbf{\degreeu}$ in Section \ref{sec:preliminaries} are still satisfied.
Then, following Equation \eqref{eq:edge_vertex_spaces}, we will also define 
\begin{equation}
	\degreeu_\edge := \max_{\overline{\face} \supset \edge} \degreeu_\face\;,\qquad
	\degreeu_\vertex := \max_{\overline{\face} \ni \vertex} \degreeu_\face\;.
\end{equation}
We will denote the sets containing vertical and horizontal interior edges of $\mesh$ with $\meshInteriorEV$ and $\meshInteriorEV$, respectively.
Finally, for a vertex $\vertex$, we define orders of smoothness $\bsmooth_h(\vertex)$ and $\bsmooth_v(\vertex)$ as follows,
\begin{equation*}
	\bsmooth_h(\vertex) := \min_{\substack{\edge \ni \vertex\\\edge \in \meshEV}} \bsmooth(\edge)\;,\qquad
	\bsmooth_v(\vertex) := \min_{\substack{\edge \ni \vertex\\\edge \in \meshEH}} \bsmooth(\edge)\;.
\end{equation*}
We start by defining the segments of the T-mesh as connected unions of horizontal or vertical edges that have the same associated $\degreeu_{\edge}$.
\begin{definition}[Segments of the T-mesh]
	Let $A \subseteq \meshInteriorE\cup\meshV$ be a finite set of either horizontal or vertical edges $\tau\in\meshInteriorE$,
	\begin{itemize}
		\item $L_A := \cup_{\edge \in A}\edge$ is non-empty and connected,
		\item $\degreeu_\edge = \degreeu_{\edge'} =: \degreeu_A$ for any $\edge, \edge' \in A \cap \meshInteriorE$.
	\end{itemize}
	Then $A$ will be called a (horizontal or vertical) segment.
\end{definition}


Lemma \ref{lem:tmesh_segment}, which will be presented shortly, identifies sufficient conditions allowing the smoothness across a segment of the T-mesh to be reduced while preserving \ginoun{}.
The next three results discuss the dimensions of spaces of univariate polynomials, and Example \ref{ex:univ_dim} presents an application of these results; the proof of Lemma \ref{lem:tmesh_segment} will use these results as well.

In the following, $\overline{\PP}_\degreeu$ is used to denote the vector space of univariate polynomials in variable $x$ of degree at most $\degreeu$; $\overline{\PP}_\degreeu := 0$ for $\degreeu < 0$.
Finally, for $I = \{1, \dots, k\}$ and some $a_i \in \RR$ and $d_i \in \ZZ_{\geq -1}$, $i \in I$, we define $\smin{I} \subseteq I$ to be the largest set such that
\begin{itemize}
	\item all $a_i$, $i \in \smin{I}$, are distinct;
	\item for each $i \in \smin{I}$, $d_i = \min \left\{ d_j~:~ a_i = a_j\;, j \in I  \right\}$.
\end{itemize}

\begin{lemma}[Proposition 1.8, Mourrain \cite{mourrain2014dimension}]
	\label{lem:apolar_uniform}
	For $i \in I = \{1, \dots, k\}$, let $a_i \in \RR$ and $d_i \in \ZZ_{\geq -1}$.
	Consider linear polynomials $\ell_i = x-a_i$, $i \in I$, and define the vector space $V$ as
	\begin{equation}
		V := \sum_{i \in I} \ell_i^{d_i}\overline{\PP}_{\degreeu-d_i}\;.
	\end{equation}
	Then, the dimension of $V$ is given by the following formula,
	\begin{equation}
		\dimwp{V} = \min\left( \degreeu+1,~\sum_{i \in \smin{I}} (\degreeu-d_i+1)_+ \right)\;.
	\end{equation}
\end{lemma}

\begin{lemma}
	\label{lem:apolar_nonuniform}
	For $i \in I = \{1, \dots, k\}$, let $a_i \in \RR$, $d_i \in \ZZ_{\geq -1}$ and $e_i \in \ZZP$.
	Consider linear polynomials $\ell_i = x-a_i$, $i \in I$, and define the vector space $V$ as below,
	\begin{equation}
		V := \sum_{i \in I} \ell_i^{d_i}\overline{\PP}_{\degreeu-d_i-e_i}\;.
	\end{equation}
	Then, the dimension of $V$ is given by the following formula,
	\begin{equation}
	\begin{split}
		\dimwp{V} = \sum_{j=0}^d
		&\left[
		\min\left( \degreeu-e^{j+1}+1,~\sum_{i \in \smin{I^j}} (\degreeu-e^{j+1}-d_i+1)_+ \right) \right.\\
		&\qquad\quad \left.- \min\left( \degreeu-e^{j}+1,~\sum_{i \in \smin{I^j}} (\degreeu-e^{j}-d_i+1)_+ \right)
		\right]\;,
	\end{split}
	\end{equation}
	where we use the following definitions,
	\begin{equation}
	\begin{split}
	&E := \{ e_1, \dots, e_k \}\;,\;
	d := \#E-1\;,\\
	&e^j := \begin{dcases}
	\degreeu+1\;,& j = 0\;,\\
	\max E \backslash \{e^0,\dots,e^{j-1}\}\;, & j=1, \dots, d+1,
	\end{dcases}\\
	& I^j := \big\{  i \in I~:~e_i < e^j \big\}\;,\; j= 0, \dots, d+1\;.
	\end{split}
	\end{equation}
\end{lemma}
\begin{proof}
	The proof follows from Lemma \ref{lem:apolar_uniform} and \cite[Lemma 4.5]{toshniwal2019polynomial}, where an analogous claim was shown for bivariate polynomials.
\end{proof}

\begin{corollary}\label{cor:apolar_full_ring}
	For $i \in I = \{1, \dots, k\}$, let $a_i \in \RR$, $d_i \in \ZZ_{\geq -1}$ and $e_i \in \ZZP$.
	Consider linear polynomials $\ell_i = x-a_i$, $i \in I$, and define the vector space $V$ as below,
	\begin{equation}
	V := \sum_{i \in I} \ell_i^{d_i}\overline{\PP}_{\degreeu-d_i-e_i}\;.
	\end{equation}
	If $\dimwp{V} = \degreeu+1$, then $V = \overline{\PP}_\degreeu$.
\end{corollary}

\begin{example}\label{ex:univ_dim}
	As an illustration of Lemma \ref{lem:apolar_nonuniform}, let us consider the vector space $V$ defined by choosing
	\begin{equation}
	\begin{split}
		&\degreeu = 3\;,\;\;
		I = \{1, 2, 3, 4\}\;,\;\;
		(d_1, d_2, d_3, d_4) = (3, 2, 3, 3)\;,\;\;\\
		&(e_1, e_2, e_3, e_4) = (0, 1, 0, 0)\;,\;\;
		(a_1, a_2, a_3, a_4) = (-1, 0, 0, 1)\;.
	\end{split}
	\end{equation}
	That is, we choose $V$ as the following vector space, where $\ell_i = x - a_i$,
	\begin{equation}
		V = \ell_1^3 \overline{\PP}_{0} + \ell_2^2 \overline{\PP}_{0} + \ell_3^3 \overline{\PP}_{0} + \ell_4^3 \overline{\PP}_{0}\;.
	\end{equation}
	Then, following the definitions in Lemma \ref{lem:apolar_nonuniform}, we have
	\begin{equation}
	\begin{split}
		&E = \{0, 1\}\;,\;\;
		d = 1\;,\;\;
		(e^0,e^1, e^2) = (4, 1,  0)\;,\\
		& I = \{1, 2, 3, 4\}\;,\;\;
		I^0 = I\;,\;\;
		I^1 = \{ 1, 3, 4 \}\;,\;\;
		I^2 =\emptyset\;.
	\end{split}
	\end{equation}
	Therefore, we see that,
	\begin{equation}
		M(I^0) = \{1, 2, 4\}\;,\;\;
		M(I^1) = \{1, 3, 4\}\;,
	\end{equation}
	and the dimension of $V$ follows as
	\begin{equation}
	\small
	\begin{split}
		\dimwp{V} 
		&= \min\left( 3-1+1,~\sum_{i \in \smin{I^0}} (3-1-d_i+1)_+ \right) - \min\left( 3-4+1,~\sum_{i \in \smin{I^0}} (3-4-d_i+1)_+ \right)\\
		& + \min\left( 3-0+1,~\sum_{i \in \smin{I^1}} (3-0-d_i+1)_+ \right) - \min\left(3-1+1,~\sum_{i \in \smin{I^1}} (3-1-d_i+1)_+ \right)\\
		&= (1-0) + (3-0) = 4 = \dimwp{\overline{\PP}_3}\;,
	\end{split}
	\end{equation}
	so that $V = \overline{\PP}_3$.
\end{example}

\begin{definition}[Weight of a segment]\label{def:segment_weight}
	Given a segment $A$, define the set $T$ as
	\begin{equation}
		T = \left\{ \overline{\edge} \in\meshE~:~ \overline{\edge} \text{ intersects } L_A \text{ trasversally}\right\}\;.
	\end{equation}
	Let $a_{\overline{\edge}}$ be the horizontal (resp. vertical) coordinate for the vertical (resp. horizontal) edge $\overline{\edge}$.
	Then, the weight of $A$, $\omega^\bsmooth(A)$, is defined as
	\begin{equation*}
	\omega^\bsmooth(A) := 
		\dimwp{\sum_{\overline{\edge} \in T} (x - a_{\overline{\edge}})^{\bsmooth(\overline{\edge})+1}\overline{\PP}_{\degreeu_A - \max(0,\degreeu_A-\degreeu_{\overline{\edge}})-\bsmooth(\overline{\edge})-1}}\;.
	\end{equation*}
	Note that $\omega^\bsmooth(A)$ can be computed by a direct application of Lemma \ref{lem:apolar_nonuniform}.
\end{definition}

\begin{lemma}\label{lem:tmesh_segment}
	Let $\bsmooth$ be such that $\quotientComplex^\bsmooth$ is \gi{}, and consider a segment $A$.
	Let the smoothness distribution $\bsmoothr$ be defined as follows,
	\[\bsmoothr(\edge)=\begin{cases}
	\bsmooth(\edge)\; & \text{for $\edge \notin A$},\\
	\smooth \; & \text{otherwise}, 
	\end{cases}\]
	where $r\leq \bsmooth(\tau)$ for all for all $\edge \in A$, $\smooth \in \ZZ_{\geq -1}$.
	If either one of the following two requirements is satisfied,
	\begin{enumerate}[label=(\alph*)]
		\item $A \subsetneq B$ for some segment $B$,  and $\bsmoothr(\edge) \leq \smooth$ for all $\edge \in B$,
		\item $\omega^{\bsmoothr}(A) = \degreeu_A+1$,
	\end{enumerate}
	then $\quotientComplex^\bsmoothr$ is also \gi{}.
\end{lemma}
\begin{remark}
	Since the weight of a segment $A$ depends only on the smoothness of edges transversal to $A$, reducing the smoothness across $A$ does not change its weight.
	That is, in the statement of Lemma \ref{lem:tmesh_segment}, we always have $\omega^{\bsmoothr}(A) = \omega^{\bsmooth}(A)$.
\end{remark}
\begin{proof}
	Using Lemma \ref{lem:support}, we can study the \ginoun{} of $\quotientComplex^\bsmoothr$ by studying $H_0(\idealComplex^{\bsmoothr} / \idealComplex^{\bsmooth})$ on the segment $A$.
	This is essentially a one dimensional problem.
	Consider then the horizontal segment $A$ as shown below; the proof for vertical segments is analogous.
	
	
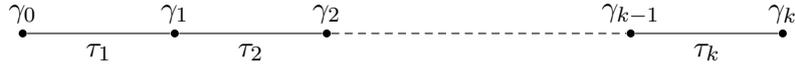
\begin{center}
		\begin{tikzpicture}[scale=2]
		\node[circle,fill=black,black,inner sep=0pt,minimum size=3pt] (c0) at (0,0) {};
		\node[circle,fill=black,black,inner sep=0pt,minimum size=3pt] (c1) at (1,0) {};
		\node[circle,fill=black,black,inner sep=0pt,minimum size=3pt] (c2) at (2,0) {};
		\node[circle,fill=black,black,inner sep=0pt,minimum size=3pt] (c3) at (4,0) {};
		\node[circle,fill=black,black,inner sep=0pt,minimum size=3pt] (c4) at (5,0) {};
		\draw[thin] (c0) -- (c1) -- (c2);
		\draw[thin, densely dashed] (c2) -- (c3);
		\draw[thin] (c3) -- (c4);
		\node[below] at ($(c0)!0.5!(c1)$) {$\edge_1$};
		\node[below] at ($(c1)!0.5!(c2)$) {$\edge_2$};
		\node[below] at ($(c3)!0.5!(c4)$) {$\edge_k$};		
		\node[above] at (c0) {$\vertex_0$};
		\node[above] at (c1) {$\vertex_1$};
		\node[above] at (c2) {$\vertex_2$};
		\node[above] at (c3) {$\vertex_{k-1}$};
		\node[above] at (c4) {$\vertex_k$};
		\end{tikzpicture}
	\end{center}
	
	$A$ contains the edges $\edge_1, \dots, \edge_k\in\meshInteriorE $ and the vertices $\vertex_0, \vertex_1, \dots, \vertex_k$.
	By definition, $A$ contains at least two different vertices, and $\degreeu_{\edge} =: \degreeu_A$ for all edges $\edge \in A$.
	Let $T$ be the set from Definition \ref{def:segment_weight}, i.e., the set containing all vertical edges that intersect $L_A$.
	Note that $\idealComplex^{\bsmoothr} / \idealComplex^{\bsmooth}$ is not supported on any $\overline{\edge} \in T$.
	
	When condition (a) is satisfied, the proof is very simple for the following reason.
	Firstly, since $A \subsetneq B$, we must have $\degreeu_A = \degreeu_B$ from the definition of segments.
	Without loss of generality, let $\vertex_0 \in \interior{L}_B \cap \boundary L_A$.
	Then, $\ideal^\bsmooth_{\vertex_0} = \ideal^\bsmoothr_{\vertex_0}$ and, as a consequence, $\idealComplex^\bsmoothr / \idealComplex^\bsmooth$ is not supported on $\vertex_0$.
	Every element of $\ideal^\bsmoothr_{\vertex_i}$ can be expressed as a sum of elements of $\ideal^\bsmoothr_{\overline{\edge}}$, $\vertex_i \in \overline{\edge} \in T$, and $\ideal^\bsmoothr_{\vertex_0}$.
	As a result, $H_0(\idealComplex^\bsmoothr / \idealComplex^\bsmooth)$ vanishes.
	
	Let us now examine condition (b).
	Let $\ell_A$ be a non-zero linear polynomial that vanishes on $L_A$.
	By definition, for all $\edge \in A$,
	\begin{equation}
	\ideal_{\edge}^{\bsmoothr} = \bigl\{ \ell_A^{\smooth+1}f \colon f\in \PP_{\degreeu_A(\degreeu_A-\smooth-1)}\bigr\} =: \ideal_A^\bsmoothr\;.
	\end{equation}
	Let $\ell_{\overline{\edge}}$ be a non-zero linear polynomial that vanishes on vertical edge $\overline{\edge} \in T$.
	Since $\omega^\bsmoothr(A) = \degreeu_A+1$, we can use Corollary \ref{cor:apolar_full_ring} to write
	\begin{equation*}
		\ideal_A^{\bsmoothr}
			= \ell_A^{r+1}\sum_{\overline{\edge}\in T} \ell_{\overline{\edge}}^{\bsmoothr(\overline{\edge})+1}\PP_{(\degreeu_A-\bsmoothr(\overline{\edge})-1)(\degreeu_A-r-1)}\;.
	\end{equation*}
	Then, for any $i$, any element of $\ideal_{\vertex_i}^\bsmoothr$ can be written as the sum of elements of $\ideal_{\overline{\edge}}^\bsmoothr$, $\overline{\edge} \in T$.
	Since $\idealComplex^{\bsmoothr} / \idealComplex^\bsmooth$ is not supported on any $\overline{\edge}$, $H_0(\idealComplex^{\bsmoothr} / \idealComplex^\bsmooth) = 0$.
	
	For both conditions (a) and (b), the claim of \ginoun{} of $\quotientComplex^\bsmoothr$ follows from the above and Corollary \ref{cor:new_dimension}.
\end{proof}

Lemma \ref{lem:tmesh_segment} discusses the setting when the smoothness is reduced across a single segment of the mesh.
Its successive applications can help us compute the dimension of a large class of splines on $\mesh$ with mixed smoothness.
The next result is immediate and is completely analogous to Lemma \ref{lem:tri_edge} which was shown for triangulations.
Its statement is simple: if an edge intersects another edge, and the order of smoothness across the latter is $-1$, then we can also reduce the order of smoothness across the former to $-1$ while preserving \ginoun{}.

\begin{lemma}\label{lem:tmesh_edge}
	Let $\bsmooth$ be such that $\quotientComplex^\bsmooth$ is \gi{}, and consider an interior edge $\edge \in \meshInteriorE$.
	If $\edge'$ is another edge such that $\edge \cap \edge'$ is not empty, $\bsmooth(\edge') = -1$, and $\degreeu_{\edge'} \geq \degreeu_\edge$, then $\quotientComplex^\bsmoothr$ is \gi{} where
	\begin{equation}
		\bsmoothr(\edge'')=
		\begin{cases}
			\bsmooth(\edge'')\; & \edge'' \neq \edge,\\
			-1 \; & \text{otherwise}.
		\end{cases}
	\end{equation}
\end{lemma}

We are now in a position to present examples where we compute the dimensions of spline spaces on T-meshes that contain holes.
The approach will be exactly analogous to the one taken in the previous sub-section, i.e., we will try to see if the spline space can be interpreted as a pruned spline space on a T-mesh without holes.

\begin{figure}[t]
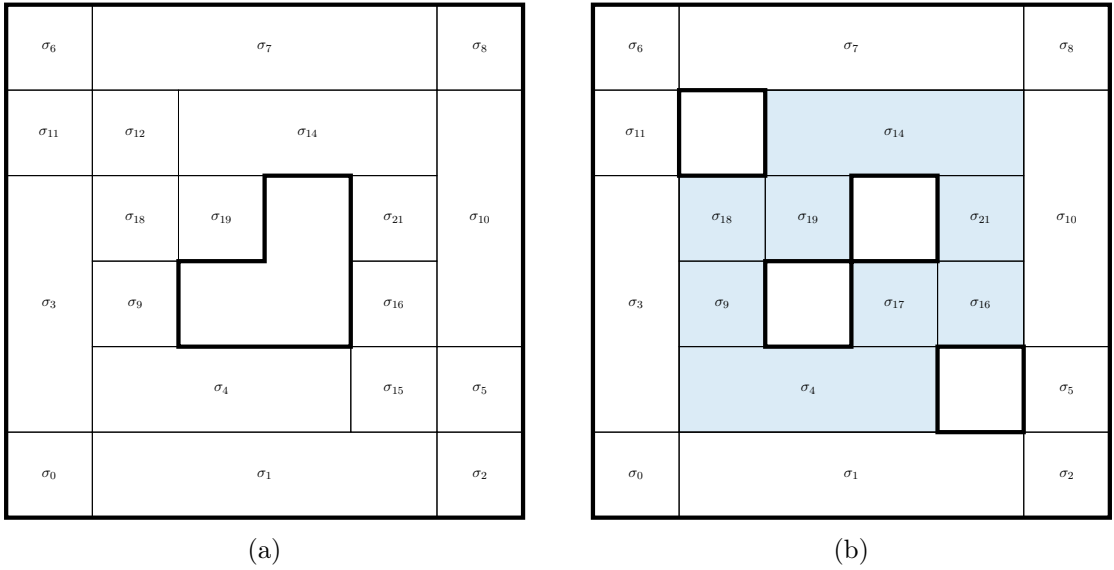

	\centering
	\subcaptionbox{}[0.49\textwidth]{\includetikz{tmesh_tmesh_hole_2}{./tikz}}
	\subcaptionbox{}[0.49\textwidth]{\includetikz{tmesh_tmesh_hole_nud_2}{./tikz}}
	\caption{Example \ref{ex:tmesh} shows how the dimension of $C^1$ splines on the above meshes in figures (a) and (b) can be computed.
	All mesh faces have been labelled, and the boundaries of the respective domains have been displayed in bold; as is clear, the domains are not simply connected.
	On both meshes, we are interested in splines whose pieces are bi-quadratic polynomials on the faces without colour, and bi-cubic polynomials otherwise. (Continues in next figure.)
}
\end{figure}
\begin{figure}\ContinuedFloat
	\centering
	\subcaptionbox{}[0.49\textwidth]{\includetikz{tmesh_tmesh_hole_0}{./tikz}}
	\subcaptionbox{}[0.49\textwidth]{\includetikz{tmesh_tmesh_hole_nud_0}{./tikz}}\\
	\subcaptionbox{}[0.49\textwidth]{\includetikz{tmesh_tmesh_hole_1}{./tikz}}
	\subcaptionbox{}[0.49\textwidth]{\includetikz{tmesh_tmesh_hole_nud_1}{./tikz}}
	\caption{
		(Continues from previous figure.)
		Dimension of splines in (a) and (b) can be easily computed as follows.
		First, we use the results from \cite{mourrain2014dimension} and \cite{toshniwal_polynomial_2019} to get the dimension of splines on meshes in (c) and (d), respectively.
		Thereafter, we use Corollary \ref{cor:new_dimension} to compute the dimension when smoothness across the dashed edges in (e) and (f) has been reduced to $\bsmoothr(\edge) = -1$.
		This decouples the faces inside the dashed region from those outside the dashed region.
		Then, the required dimension can be computed using Theorem \ref{thm:hole_dimension}, i.e., by subtracting $\dimwp{\PP_\face}$ for each $\face$ contained inside the dashed region.
	}\label{fig:tmesh_hole}
\end{figure}

\begin{example}[$C^1$ splines on domains with holes]\label{ex:tmesh}
	Consider the problem of building $C^1$ splines on the two domains shown in Figure \ref{fig:tmesh_hole}(a) and (b).
	On the mesh in (a), we are interested in splines that are biquadratic polynomials when restricted to any mesh face.
	On the mesh in (b), on the other hand, we are interested in splines that are biquadratic polynomials restricted to the white faces, and bicubic polynomials when restricted to the blue faces.
	Proceeding as in the case of triangulations, we will compute the dimension of such spline spaces by interpreting them as pruned spline spaces on the meshes in (c) and (d).
	
	The dimension of splines on the mesh in (c) can be computed using the results from \cite{mourrain2014dimension}, while the dimension on the mesh in (d) can be computed using \cite{toshniwal_polynomial_2019}.
	Both are computed to be $30$ and $50$, respectively.
	Then, Lemma \ref{lem:tmesh_segment} allows us to see that, for the mesh in (c), we can reduce the smoothness across any segment that is composed of at least two edges.
	For the mesh in (d), Lemma \ref{lem:tmesh_segment} allows us to reduce the smoothness across any edge of the mesh.
	Then, a combination of Lemma \ref{lem:tmesh_segment} and \ref{lem:tmesh_edge} allows us to reduce the smoothness across all dashed edges in (e) and (f) to $-1$.
	The dimensions of the resulting spline spaces are $58$ and $118$.
	Thereafter, we can use Theorem \ref{thm:hole_dimension} to compute the dimension of splines on the meshes in (a) and (b), respectively, as $31$ and $54$; c.f. the supplementary M2 scripts provided with this paper.
	(Also, recall Remark \ref{rem:dim_increase}.)
\end{example}

	\section{Conclusions}

Piecewise-polynomial splines are extensively applied in the fields such as computer-aided geometric design \cite{farin2002handbook} and numerical analysis \cite{hughes2005isogeometric}.
In practice, {spline refinements} -- degree elevation, smoothness reduction, mesh subdivision -- are some of the most important {exact} operations that can help enhance the approximation power of the spline space.
The ability to do so in a local manner is central to efficient applications of splines.

In this paper, we study the problem of dimension computation for splines while focusing on local smoothness reduction.
Starting from smooth splines on arbitrary polygonal meshes of $\RR^2$, we derive results that allow us to compute the exact spline space dimension once smoothness requirements across a subset of the mesh edges have been relaxed.
The derived results are very widely applicable, and in order to explore their specific implications we restrict our focus -- we show how they can be used to compute the dimension of splines on polgyonal meshes, triangulations and T-meshes of arbitrary topology.

For instance, our results allow us to compute the dimension of splines on T-meshes with holes; see Figure \ref{fig:tmesh_hole}.
Such splines can then be used to build smooth surfaces of arbitrary topologies -- a major application of splines in geometric modelling and numerical analysis \cite{toshniwal2017smooth}.
The construction of a spline-space basis that can handle such applications is the focus of ongoing research.

Another particularly interesting research direction is the one opposite to the one we study here -- the \emph{inexact} operation of \emph{spline coarsening}.
This is also very useful in practice.
For example, when numerically solving a PDE, if a complex solution feature simplifies over time, one would also want to reduce the spline space's approximation power for efficiency.
Studying this setting -- where a richer spline space is used to compute the dimension of its subspace -- will likely require an entirely different approach than the one adopted here.

%
%

	\bibliographystyle{plain}
	\bibliography{bibliography}

\end{document}